\documentclass[reqno,12pt,twoside]{article}
\usepackage[a4paper,top=23mm,bottom=23mm,inner=23mm,outer=21mm]{geometry}
\usepackage{amssymb}
\usepackage{amsmath}
\usepackage{amsthm}
\numberwithin{equation}{section}

\newtheorem{theorem}{Theorem}[section]
\newtheorem{lemma}[theorem]{Lemma}
\newcommand{\Nzero}{\mathbb N_0}
\newcommand{\eps}{\varepsilon}
\newcommand{\cP}{\mathcal P}
\begin{document}

\title{\Large On a problem of Nathanson related to strongly minimal asymptotic bases of order $h$}
\author{\large Shi-Qiang Chen\thanks{This work is supported by the National Natural Science Foundation of China (Grant No. 12301003), the Anhui Provincial Natural Science Foundation (Grant No. 2308085QA02) and the University Natural Science Research Project of Anhui Province (Grant No. 2022AH050171).}}
\date{} \maketitle
\vskip -3cm
\begin{center}
\vskip -1cm { \small
\begin{center}
School of Mathematics and Statistics, Anhui Normal University
\end{center}
\begin{center}
Wuhu 241002, PR China
\end{center}}
\end{center}

{\bf Abstract.} Let $h\geq2$ be an integer and let $1/h<\theta\leq1/(h-1)$. In this paper, we prove that there exists a strongly minimal asymptotic basis $A$ of order $h$ such that $A(x)\asymp x^\theta$. This solves a problem posed by Nathanson in 1988.



\section{Introduction}
Let $\Nzero$ denote the set of nonnegative integers. For $A\subseteq\Nzero$, a real number $x\geq1$, and a positive integer $h$, define
\[
 A(x)=|A\cap[0,x]|
\]
and
\[
 hA=\{a_1+\cdots+a_h:a_1,\ldots,a_h\in A\}.
\]
We call $A$ an asymptotic basis of order $h$ if $hA$ contains every sufficiently large integer. For $a\in A$, put
\[
 E_a=hA\setminus h(A\setminus\{a\}),\qquad E_a(x)=|E_a\cap[0,x]|.
\]
Following Nathanson~\cite{Nathanson1988}, an asymptotic basis $A$ is called \emph{strongly minimal} if $E_a(x)\gg_a A(x)^{h-1}$ for any $a\in A$. 

If $A$ is an asymptotic basis of order $h$, then
\[
 x-O(1)\leq \binom{A(x)+h-1}{h},
\]
and hence $A(x)\gg x^{1/h}$. If $A$ is strongly minimal, then for every $a\in A$,
\[
 A(x)^{h-1}\ll_a E_a(x)\leq x+1,
\]
and hence $A(x)\ll x^{1/(h-1)}$.

Nathanson~\cite{Nathanson1988} constructed strongly minimal asymptotic bases of order $h$ with $A(x)\asymp x^{1/h}$ and posed the following problem.\\
{\bf Nathanson's Problem.}
Let $h\geq2$ be an integer and let $1/h<\theta\leq1/(h-1)$. Does there exist a strongly minimal asymptotic basis $A$ of order $h$ and a constant $c>0$ such that
\[
 A(x)>cx^\theta
\]
for all sufficiently large $x$?

For classical results on minimal and strongly minimal asymptotic bases, see
\cite{ChenChen2011,Chen2025Strongly,ErdosNathanson1980,JiaNathanson1989,Nathanson1974,Sun2021}.
For closely related work on thin, sparse, and extremal additive bases, see
\cite{DeshouillersFouvry1976,ErdosNathanson1977,Pilatte2024,Vu2000,Wooley2003}.

In this paper, we solve this problem in the stronger form $A(x)\asymp x^\theta$.

\begin{theorem}\label{thm11}
Let $h\ge2$ be an integer and let $1/h<\theta\le1/(h-1)$. Then there exists a strongly minimal asymptotic basis $A$ of order $h$ such that $A(x)\asymp x^\theta$.
\end{theorem}

Throughout this paper, put $p=h-1$ and $\kappa=p\theta$.

\section{Lemmas}

For a positive integer $m$, let $\mathbb Z_m$ denote the set of residue classes modulo $m$,
represented by $0,1,\ldots,m-1$; all equalities in $\mathbb Z_m$ are understood modulo $m$.
Whenever intervals occur in counting arguments or additive bases, they are intersected with $\mathbb Z$ unless explicitly declared to be real intervals.
For a closed interval $I=[u,v]$, write $\operatorname{mid}(I)=(u+v)/2$.
For a nonempty bounded integer interval $K$, put
$\partial K=\{\min K,\max K\}$, and for a nonempty set $S\subseteq K$ define
$\operatorname{dist}(S,\partial K)=\min\{|x-y|:x\in S,\ y\in\partial K\}$.
For a real interval $I$, let $\operatorname{int}I$ denote its interior.
The length of a real interval is the difference of its endpoints. 

Let
\[
 e_i\in\Nzero,\qquad e_i\ne e_j\ (i\ne j),\qquad
 q_i\ne q_j\ (i\ne j),\qquad q_i\ \text{prime},
\]
and put
\[
 Q_i=\prod_{\nu\le i}q_\nu.
\]
We require
\begin{equation}\label{eq21}
 q_i>q_0,\qquad q_i^{-1/h}<2^{-i-10},\qquad
 q_{i+1}>C_*Q_i^4.
\end{equation}
The parameters will be fixed in Section~\ref{sec3} in the order
$R,\sigma,\zeta,\eta,q_0,C_*$.

For $N\ge1$ and $L,T>0$, put
\[
 B_N:=\{e_1,\ldots,e_N\}\subseteq[0,L],
\]
and require
\begin{equation}\label{eq22}
 B_N\cap(e_i+q_i\mathbb Z)=\{e_i\}
 \qquad(1\le i\le N).
\end{equation}
If
\[
 \{i\le N:Q_i\le\eta T\}\ne\varnothing,
\]
define
\begin{equation}\label{eq23}
 k=\max\{i\le N:Q_i\le\eta T\},\qquad Q=Q_k.
\end{equation}
For
\[
 X\subseteq[0,L]\setminus B_N,
\]
we require every $x\in X$ to satisfy
\begin{equation}\label{eq24}
\begin{aligned}
 x&\not\equiv e_i\pmod{q_i}\quad(1\le i\le k),\\
 x&\not\equiv e_{k+1}\pmod{q_{k+1}}
 \quad(k<N,\ q_{k+1}\le2L).
\end{aligned}
\end{equation}

\begin{lemma}\label{lem21}
Let $R,C_*>1$, let $L$ be sufficiently large, let $T=\lfloor\sigma L^\kappa\rfloor$, and suppose that $Q_1\le\eta T\le L$.
Assume that $q_1,\ldots,q_N$ satisfy \eqref{eq21}, and let $k,Q$ be given by \eqref{eq23}. Then
\begin{equation}\label{eq25}
 Q_k>2^{h\sum_{i=1}^k(i+10)}\ge2^{(h/2)k^2},\qquad
 \prod_{i=1}^N(1+q_i^{-1/h})<e^{2^{-10}}.
\end{equation}
Moreover,
\[
 k=O(\sqrt{\log L}),\qquad
 \sum_{i=1}^N\frac1{q_i}<2^{-10}.
\]
If $k<N$, then
\[
 q_{k+1}>C_*^{1/5}(\eta T)^{4/5};
\]
and if $k+2\le N$, then
\[
 q_{k+2}>C_*\eta^4T^4>2L.
\]
Furthermore, if $A_0=\{e_1,\ldots,e_N\}\subseteq[0,L]$ satisfies \eqref{eq22}, and every element of
$X\subseteq[0,L]\setminus A_0$ satisfies \eqref{eq24}, then
\[
 (A_0\cup X)\cap(e_i+q_i\mathbb Z)=\{e_i\}
 \qquad(1\le i\le N).
\]
If $B\subseteq[0,L]$, $e\in B$, and $q>RL$, then
\[
 B\cap(e+q\mathbb Z)=\{e\}.
\]
\end{lemma}

\begin{proof}
By \eqref{eq21}, for every $i$ we have
\[
 q_i^{-1/h}<2^{-i-10}.
\]
Hence $q_i>2^{h(i+10)}$, and therefore
\[
 Q_k=\prod_{i=1}^kq_i
 >2^{h\sum_{i=1}^k(i+10)}
 \ge2^{(h/2)k^2}.
\]
Since $h\ge2$,
\[
 \sum_{i=1}^N\frac1{q_i}
 \le\sum_{i=1}^Nq_i^{-1/h}
 <\sum_{i\ge1}2^{-i-10}=2^{-10},
\]
and, for every choice of the first $N$ primes,
\[
 \prod_{i=1}^N(1+q_i^{-1/h})
 \le\exp\left(\sum_{i=1}^Nq_i^{-1/h}\right)
 <e^{2^{-10}}.
\]
Thus the bounds in \eqref{eq25} are independent of $N$, and
$\prod_{i\ge1}(1+q_i^{-1/h})<\infty$.
Since $Q_k\le\eta T\le L$,
\[
 2^{(h/2)k^2}<Q_k\le L.
\]
Taking logarithms gives $k=O(\sqrt{\log L})$.

If $k<N$, then by the maximality of $k$ and \eqref{eq21},
\[
 \eta T<Q_{k+1}=Q_kq_{k+1},\qquad
 q_{k+1}>C_*Q_k^4,
\]
and hence
\[
 \eta T<Q_kq_{k+1}
 <C_*^{-1/4}q_{k+1}^{5/4}.
\]
It follows that $q_{k+1}>C_*^{1/5}(\eta T)^{4/5}$.
If $k+2\le N$, then $Q_{k+1}>\eta T$, and so
\[
 q_{k+2}>C_*Q_{k+1}^4>C_*\eta^4T^4>2L.
\]
The last inequality holds for all sufficiently large $L$, since
\[
 T\asymp L^\kappa,\qquad
 4\kappa>4p/h\ge2.
\]

Let $x\in X$. If $i\le k$, then \eqref{eq24} gives
\[
 x\not\equiv e_i\pmod{q_i}.
\]
If $i=k+1\le N$ and $q_i\le2L$, the same conclusion follows again from \eqref{eq24}.
In all remaining cases,
\[
 i=k+1\Longrightarrow q_i>2L,
 \qquad
 i\ge k+2\Longrightarrow q_i\ge q_{k+2}>2L,
\]
where the second implication follows from \eqref{eq21} and the preceding inequality. Hence
\[
 x,e_i\in[0,L],\qquad |x-e_i|\le L<q_i.
\]
If $x\equiv e_i\pmod{q_i}$, then $q_i\mid(x-e_i)$; together with
$|x-e_i|<q_i$, this gives $x=e_i$, contradicting $X\cap A_0=\varnothing$.
Combining this with \eqref{eq22} for $A_0$, we obtain
\[
 (A_0\cup X)\cap(e_i+q_i\mathbb Z)=\{e_i\}
 \qquad(1\le i\le N).
\]
Finally, suppose that $x,e\in B\subseteq[0,L]$ and $q>RL>L$.
If $x\equiv e\pmod q$, then
\[
 q\mid(x-e),\qquad |x-e|\le L<q.
\]
Therefore $x=e$, which proves the last assertion.
\end{proof}

\begin{lemma}\label{lem22}
Let $Q\ge1$ and $m\ge2$ be integers, let $\ell$ be an odd prime with $(Q,\ell)=1$, and let $x_1,\ldots,x_m,b\in\mathbb Z$. Then there exist
$t_1,\ldots,t_m\in\mathbb Z$ such that
\begin{equation}\label{eq26}
 \sum_{i=1}^m t_i=0,\qquad |t_i|\le m+1,\qquad
 x_i+t_iQ\not\equiv b\pmod\ell\quad(1\le i\le m).
\end{equation}
Consequently, for arbitrary finite sets $F_1,\ldots,F_m\subseteq\mathbb Z$, if
\[
 F_i^\dagger=
 \{x+tQ:x\in F_i,\ t\in\mathbb Z,\ |t|\le m+1,\ x+tQ\not\equiv b\pmod\ell\},
\]
then
\begin{align}
 F_1+\cdots+F_m&\subseteq F_1^\dagger+\cdots+F_m^\dagger,\label{eq27}\\
 |F_i^\dagger|&\le(2m+3)|F_i|.\label{eq28}
\end{align}
Moreover, for every $y\in F_1^\dagger+\cdots+F_m^\dagger$, there exists
$y_0\in F_1+\cdots+F_m$ such that
\begin{equation}\label{eq29}
 |y-y_0|\le m(m+1)Q.
\end{equation}
\end{lemma}

\begin{proof}
Since $(Q,\ell)=1$, there exists an integer
$Q^{-1}\in[0,\ell-1]\cap\mathbb Z$ such that
\[
 QQ^{-1}\equiv1\pmod\ell.
\]
For each $1\le i\le m$, let $c_i\in[0,\ell-1]\cap\mathbb Z$ satisfy
\[
 c_i\equiv(b-x_i)Q^{-1}\pmod\ell.
\]
Since $\ell>2$, the residues of $0$ and $1$ modulo $\ell$ are distinct. Hence, for $1\le i\le m-2$,
\[
 \#\bigl(\{0,1\}\cap(c_i+\ell\mathbb Z)\bigr)\le1.
\]
Thus we may choose successively
\[
 t_i\in\{0,1\}\setminus(c_i+\ell\mathbb Z)
 \quad(1\le i\le m-2),
 \qquad
 s=\sum_{i=1}^{m-2}t_i,
 \qquad
 0\le s\le m-2.
\]
Put
\[
 \mathcal U=\{0,1,2\},
\qquad
 \mathcal U_1=\mathcal U\cap(c_{m-1}+\ell\mathbb Z),
\qquad
 \mathcal U_2=\mathcal U\cap(-s-c_m+\ell\mathbb Z).
\]
Again, since $\ell>2$, the three elements of $\mathcal U$ are pairwise distinct modulo $\ell$, while each of $\mathcal U_1$ and $\mathcal U_2$ is contained in a single residue class modulo $\ell$. Therefore
\[
 |\mathcal U_1|\le1,
 \qquad
 |\mathcal U_2|\le1,
 \qquad
 |\mathcal U\setminus(\mathcal U_1\cup\mathcal U_2)|\ge1.
\]
Choose
\[
 t_{m-1}\in\mathcal U\setminus(\mathcal U_1\cup\mathcal U_2),
 \qquad
 t_m=-s-t_{m-1}.
\]
Then
\begin{align*}
 \sum_{i=1}^m t_i
 &=s+t_{m-1}-s-t_{m-1}=0,\\
 |t_i|&\le1\le m+1\quad(1\le i\le m-2),\\
 |t_{m-1}|&\le2\le m+1,\\
 |t_m|&=s+t_{m-1}\le(m-2)+2=m\le m+1,\\
 t_i&\not\equiv c_i\pmod\ell\quad(1\le i\le m),\\
 x_i+t_iQ-b
 &\equiv Q(t_i-c_i)\not\equiv0\pmod\ell.
\end{align*}
Therefore, for every $(x_1,\ldots,x_m)\in\prod_{i=1}^mF_i$, the above choice of $t_i$ satisfies
\[
 (x_1+t_1Q,\ldots,x_m+t_mQ)\in\prod_{i=1}^mF_i^\dagger.
\]
Moreover,
\[
 \sum_{i=1}^m(x_i+t_iQ)
 =\sum_{i=1}^m x_i+Q\sum_{i=1}^m t_i
 =\sum_{i=1}^m x_i,
\]
and hence
\[
 F_1+\cdots+F_m\subseteq F_1^\dagger+\cdots+F_m^\dagger.
\]
By the definition,
\[
 |F_i^\dagger|
 \le\sum_{x\in F_i}|[-m-1,m+1]\cap\mathbb Z|
 =(2m+3)|F_i|.
\]
Finally, if $y\in F_1^\dagger+\cdots+F_m^\dagger$, then we may write
\[
 y=\sum_{i=1}^m(x_i+t_iQ),
 \quad
 y_0=\sum_{i=1}^m x_i,
 \quad
 |t_i|\le m+1.
\]
Thus
\[
 |y-y_0|
 =Q\left|\sum_{i=1}^m t_i\right|
 \le Q\sum_{i=1}^m|t_i|
 \le m(m+1)Q.
\]
\end{proof}

\begin{lemma}\label{lem23}
Let $d,m\ge1$ be integers, let $q_1,\ldots,q_d\ge2$ be pairwise coprime integers, and put
\[
 Q=\prod_{\nu=1}^d q_\nu.
\]
For $1\le i\le m$ and $1\le\nu\le d$, let
\[
 S_{\nu,i}\subseteq\mathbb Z_{q_\nu},
\]
and define
\begin{equation}\label{eq210}
 R_i=
 \left\{r\in[0,Q-1]\cap\mathbb Z:
 r\bmod q_\nu\in S_{\nu,i}\ (1\le\nu\le d)\right\}.
\end{equation}
Then
\begin{equation}\label{eq211}
 |R_i|=\prod_{\nu=1}^d|S_{\nu,i}|,
\end{equation}
and
\begin{equation}\label{eq212}
 (R_1+\cdots+R_m)\bmod Q
 =\left\{u\bmod Q:
 u\bmod q_\nu\in S_{\nu,1}+\cdots+S_{\nu,m}
 \ (1\le\nu\le d)\right\}.
\end{equation}
If $c_\nu\notin S_{\nu,i}$, then
\begin{equation}\label{eq213}
 R_i\cap(c_\nu+q_\nu\mathbb Z)=\varnothing.
\end{equation}
In particular, if $j$ is fixed and
\[
 S_{\nu,1}+\cdots+S_{\nu,m}
 =\mathbb Z_{q_\nu}
 \qquad(\nu\ne j),
\]
then
\begin{equation}\label{eq214}
 (R_1+\cdots+R_m)\bmod Q
 =\left\{u\bmod Q:
 u\bmod q_j\in S_{j,1}+\cdots+S_{j,m}\right\}.
\end{equation}
\end{lemma}

\begin{proof}
By the Chinese remainder theorem, the map
\[
 \Phi:\mathbb Z_{Q}
 \longrightarrow
 \prod_{\nu=1}^d\mathbb Z_{q_\nu},
 \qquad
 \Phi(u)=(u\bmod q_1,\ldots,u\bmod q_d)
\]
is a group isomorphism. By \eqref{eq210},
\[
 \Phi(R_i\bmod Q)=\prod_{\nu=1}^d S_{\nu,i}.
\]
Since distinct integers in $[0,Q-1]\cap\mathbb Z$ represent distinct residue classes modulo $Q$,
\[
 |R_i|=|R_i\bmod Q|
 =\left|\prod_{\nu=1}^dS_{\nu,i}\right|
 =\prod_{\nu=1}^d|S_{\nu,i}|,
\]
which proves \eqref{eq211}. Since $\Phi$ preserves addition,
\begin{align*}
 \Phi((R_1+\cdots+R_m)\bmod Q)
 &=\Phi(R_1\bmod Q)+\cdots+\Phi(R_m\bmod Q)\\
 &=\prod_{\nu=1}^d
 \bigl(S_{\nu,1}+\cdots+S_{\nu,m}\bigr).
\end{align*}
Applying $\Phi^{-1}$ gives \eqref{eq212}.

If $c_\nu\notin S_{\nu,i}$, then for every $r\in R_i$ the definition gives
\[
 r\bmod q_\nu\in S_{\nu,i},
\]
and hence
\[
 r\not\equiv c_\nu\pmod{q_\nu}.
\]
Thus \eqref{eq213} holds.

Finally, if
\[
 S_{\nu,1}+\cdots+S_{\nu,m}
 =\mathbb Z_{q_\nu}
\]
for every $\nu\ne j$, then in \eqref{eq212} there is no restriction on any coordinate except the $j$th one. This gives \eqref{eq214}.
\end{proof}

\begin{lemma}\label{lem24}
Let $q\ge5$ be a prime, let $2\le d\le h$, let $c\in\mathbb Z_q$, and put
$m=\lceil q^{1/d}\rceil$. Then there exist sets
$S_0,\ldots,S_{d-1}\subseteq\mathbb Z_q\setminus\{c\}$ such that
\[
 |S_i|=m<q\quad(0\le i<d),\qquad
 S_0+\cdots+S_{d-1}=\mathbb Z_q.
\]
\end{lemma}

\begin{proof}
For $0\le i<d$, let
\[
 D_i=\{tm^i\bmod q:t\in\mathbb Z,\ 0\le t<m\}\subseteq\mathbb Z_q.
\]
Since $q\ge5$ and $d\ge2$,
\[
 m^d\ge q,\qquad 2\le m\le\lceil\sqrt q\rceil<q.
\]
Since $q$ is prime and $m<q$, if
\[
 (t-t')m^i\equiv0\pmod q,\qquad 0\le t,t'<m,
\]
then $q\nmid m$ gives $q\mid(t-t')$. Since $|t-t'|<q$, it follows that
$t=t'$, and hence $|D_i|=m$. Moreover,
\[
 D_0+\cdots+D_{d-1}
 \supseteq[0,m^d-1]\pmod q
 =\mathbb Z_q,
 \qquad |D_i|=m<q.
\]
Thus $c-D_i\ne\mathbb Z_q$. Choose $v_i\notin c-D_i$ and put
\[
 S_i=D_i+v_i,
 \qquad c\notin S_i,
 \qquad |S_i|=m.
\]
Finally,
\[
 S_0+\cdots+S_{d-1}
 =(D_0+\cdots+D_{d-1})+\sum_{i=0}^{d-1}v_i
 =\mathbb Z_q.
\]
\end{proof}

\begin{lemma}\label{lem25}
Let $m\ge2$, $M\ge2$, and $Q\ge1$ be integers, and let
$R_1,\ldots,R_m\subseteq[0,Q-1]\cap\mathbb Z$ be nonempty. Put
\[
 \Omega=(R_1+\cdots+R_m)\bmod Q,
 \qquad
 S=QM^m,
 \qquad
 F_i=R_i+QM^{i-1}[0,M-1].
\]
Then
\begin{align}
 [mQ,S-Q]\cap\{n:n\bmod Q\in\Omega\}
 &\subseteq F_1+\cdots+F_m,\label{eq215}\\
 F_1+\cdots+F_m&\subseteq[0,S+(m-1)Q],\label{eq216}\\
 |F_i|&=M|R_i|.\label{eq217}
\end{align}
Let $c,C,X>0$, let $u_i,v_i,u,v\in\mathbb Z$, and suppose that the integer intervals
$V_i=[u_i,v_i]$ and $U=[u,v]$ satisfy
\begin{equation}\label{eq218}
 |V_i|\ge cX,
 \qquad |U|\le CX,
 \qquad U+[-cX,cX]\subseteq V_1+\cdots+V_m.
\end{equation}
Let $K\ge0$ be fixed. Put
\[
 \zeta_0=\frac c8,\qquad
 0<\zeta\le\zeta_0,\qquad
 \eta_0=\min\left\{
 \frac{\zeta}{2^{m+3}(m+7)},
 \frac{c}{8(mK+4m+10)}\right\}.
\]
After $\zeta$ is fixed, if
\[
 0<\eta\le\eta_0,
 \quad Q\le\eta X,
 \quad M=\left\lfloor(\zeta X/Q)^{1/m}\right\rfloor,
\]
then $M\ge2$, and there exist $N=O_{c,C,m,K,\zeta}(1)$ families of integers
$b_{i,\nu}$ such that
\begin{align}
 F_i+b_{i,\nu}Q&\subseteq[u_i+KQ,v_i-KQ],\label{eq219}\\
 U\cap\{n:n\bmod Q\in\Omega\}
 &\subseteq
 \bigcup_{\nu=1}^N\sum_{i=1}^m(F_i+b_{i,\nu}Q),\label{eq220}\\
 \left|\bigcup_{\nu=1}^N(F_i+b_{i,\nu}Q)\right|
 &\ll M|R_i|.\label{eq221}
\end{align}
\end{lemma}

\begin{proof}
To prove \eqref{eq215}, let
\[
 n\in[mQ,S-Q],
 \qquad
 n\bmod Q\in\Omega.
\]
By the definition of $\Omega$, there exists
$(r_1,\ldots,r_m)\in\prod_{i=1}^mR_i$ such that
\[
 n\equiv\sum_{i=1}^m r_i\pmod Q.
\]
Since
\[
 0\le r_i\le Q-1,
 \qquad
 0\le\sum_{i=1}^m r_i\le m(Q-1)<mQ,
\]
and $n\equiv\sum_i r_i\pmod Q$, we have
\[
 v=\frac{n-\sum_{i=1}^m r_i}{Q}\in\mathbb Z,
\]
with
\begin{align*}
 v&\ge\frac{mQ-m(Q-1)}Q=\frac mQ>0,\\
 v&\le\frac{S-Q}{Q}=M^m-1.
\end{align*}
Hence
\[
 0\le v<M^m.
\]
By the base-$M$ expansion,
\[
 v=\sum_{i=1}^m d_iM^{i-1},
 \qquad
 0\le d_i\le M-1.
\]
Therefore
\[
 n=\sum_{i=1}^m(r_i+QM^{i-1}d_i)
 \in F_1+\cdots+F_m.
\]
By the definition of the sets $F_i$,
\[
 \max(F_1+\cdots+F_m)
 \le m(Q-1)+Q(M-1)\sum_{i=0}^{m-1}M^i.
\]
Moreover,
\begin{align*}
 m(Q-1)+Q(M-1)\sum_{i=0}^{m-1}M^i
 &=mQ-m+Q(M^m-1)\\
 &=QM^m+(m-1)Q-m\\
 &<S+(m-1)Q.
\end{align*}
If
\[
 r+QM^{i-1}d=r'+QM^{i-1}d',
 \qquad
 r,r'\in[0,Q-1],\quad 0\le d,d'<M,
\]
then reduction modulo $Q$ gives
\[
 r\equiv r'\pmod Q.
\]
Since $0\le r,r'\le Q-1$, we obtain $r=r'$, and substitution back into the equality gives $d=d'$.
Thus every element has a unique representation and
\[
 |F_i|=M|R_i|.
\]

For the placement assertion, by definition,
\[
 M=\left\lfloor(\zeta X/Q)^{1/m}\right\rfloor,
 \qquad
 S=QM^m.
\]
Since $Q\le\eta X$ and $\eta\le\eta_0<\zeta/2^m$, we have
$(\zeta X/Q)^{1/m}\ge2$, and hence
\[
 \frac12(\zeta X/Q)^{1/m}\le M\le(\zeta X/Q)^{1/m}.
\]
Consequently,
\[
 2^{-m}\zeta X\le S\le\zeta X.
\]
Also, by the definition of $\eta_0$,
\[
 \frac SQ\ge\frac{2^{-m}\zeta X}{\eta X}
 \ge8(m+7).
\]
Put
\[
 f_i^-=\min F_i,
 \qquad
 f_i^+=\max F_i.
\]
To ensure $F_i+b_iQ\subseteq V_i$, define
\[
 B_i=\left[
 \left\lceil\frac{u_i+KQ-f_i^-}{Q}\right\rceil,
 \left\lfloor\frac{v_i-KQ-f_i^+}{Q}\right\rfloor
 \right]\cap\mathbb Z.
\]
Since
\[
 f_i^+-f_i^-\le f_i^+\le S+(m-1)Q,
\]
and
\[
 |V_i|\ge cX,
 \quad S\le\zeta X,
 \quad Q\le\eta X,
\]
we obtain, using $v_i-u_i=|V_i|-1$, $f_i^+-f_i^-\le S+(m-1)Q$, and $Q\ge1$,
\[
 |B_i|\ge\frac{cX-S}{Q}-(2K+m+4)
 \ge\frac{7cX}{8Q}-(2K+m+4)>1.
\]
The last inequality follows from
$Q\le\eta_0X\le cX/[8(mK+4m+10)]$. Therefore
\[
 |B_i|\ge1.
\]
Write
\[
 B_i=[A_i,B_i^+]\cap\mathbb Z,
 \qquad
 A=\sum_{i=1}^m A_i,
 \qquad
 B=\sum_{i=1}^m B_i^+.
\]
Then
\[
 B_1+\cdots+B_m=[A,B]\cap\mathbb Z.
\]
Hence, for every $z\in[A,B]\cap\mathbb Z$, there exists
$(b_1,\ldots,b_m)\in\prod_{i=1}^mB_i$ such that
\[
 \sum_{i=1}^m b_i=z.
\]
It follows that
\[
 [Qz+mQ,Qz+S-Q]\cap\{n:n\bmod Q\in\Omega\}
 \subseteq\sum_{i=1}^m(F_i+b_iQ).
\]
By \eqref{eq216} and $S/Q\ge8(m+7)$,
\[
 \sum_{i=1}^m f_i^+
 =\max(F_1+\cdots+F_m)
 \le S+(m-1)Q\le2S.
\]
Moreover,
\begin{align*}
 QA+mQ
 &\le\sum_{i=1}^m(u_i+KQ-f_i^-+Q)+mQ\\
 &\le\sum_{i=1}^m u_i+(mK+2m)Q,\\
 QB+S-Q
 &\ge\sum_{i=1}^m(v_i-KQ-f_i^+-Q)+S-Q\\
 &\ge\sum_{i=1}^m v_i-S-(mK+m+1)Q.
\end{align*}
On the other hand, \eqref{eq218} gives
\[
 \sum_{i=1}^m u_i\le u-\lfloor cX\rfloor,\qquad
 \sum_{i=1}^m v_i\ge v+\lfloor cX\rfloor.
\]
Since $\zeta\le c/8$ and $\eta_0\le c/[8(mK+4m+10)]$,
\[
 S+(mK+2m+1)Q
 \le\left(\frac c8+\frac{c(mK+2m+1)}{8(mK+4m+10)}\right)X
 <cX.
\]
Together with $\lfloor cX\rfloor>cX-1$, this yields
\[
 QA+mQ<u+1,\qquad QB+S-Q>v-1.
\]
Since $QA+mQ,QB+S-Q,u,v$ are all integers,
\[
 QA+mQ\le u,\qquad QB+S-Q\ge v.
\]
Put
\[
 \Lambda=S-(m+1)Q,
 \qquad
 \Lambda\ge\frac12S\ge2^{-m-1}\zeta X.
\]
For
\[
 \mathcal W_z=[Qz+mQ,Qz+S-Q],
\]
we have
\[
 \mathcal W_z\cap\mathcal W_{z+1}\ne\varnothing
 \quad(\Lambda\ge Q).
\]
Since $QA+mQ\le u$, $QB+S-Q\ge v$, and adjacent intervals overlap,
\[
 U\subseteq\bigcup_{z=A}^{B}\mathcal W_z.
\]
Put
\[
 g=\left\lfloor\frac{\Lambda}{2Q}\right\rfloor,
 \qquad
 \frac{\Lambda}{3}\le Qg\le\frac{\Lambda}{2}
 \quad(\Lambda/Q\ge6).
\]
Since $QB+S-Q\ge v\ge u$, define
\[
 z_-:=\min\{z\in[A,B]\cap\mathbb Z:Qz+S-Q\ge u\},
\]
\[
 z_+:=\min\{z\in[A,B]\cap\mathbb Z:Qz+S-Q\ge v\}.
\]
Then
\[
 A\le z_-\le z_+\le B,
\]
\[
 Qz_-+mQ\le u\le Qz_-+S-Q,
 \qquad
 Qz_++mQ\le v\le Qz_++S-Q,
\]
and
\[
 0\le z_+-z_-,
 \qquad Q(z_+-z_-)\le v-u+Q.
\]
Let
\[
 N:=1+\left\lceil\frac{z_+-z_-}{g}\right\rceil,
 \qquad
 z_\nu:=\min\{z_-+(\nu-1)g,z_+\}
 \quad(1\le\nu\le N).
\]
Then
\[
 z_1=z_-,\qquad z_N=z_+,
 \qquad z_\nu\in[A,B]\cap\mathbb Z,
\]
and
\[
 0\le z_{\nu+1}-z_\nu\le g,
 \qquad
 Q(z_{\nu+1}-z_\nu)\le Qg\le\frac{\Lambda}{2}<\Lambda
 \quad(1\le\nu<N).
\]
Therefore
\[
 \mathcal W_{z_\nu}\cap\mathcal W_{z_{\nu+1}}\ne\varnothing
 \quad(1\le\nu<N),
 \qquad
 U\subseteq\bigcup_{\nu=1}^N\mathcal W_{z_\nu}.
\]
Since
\[
 z_\nu\in[A,B]\cap\mathbb Z=B_1+\cdots+B_m,
\]
for each $1\le\nu\le N$ we may choose
\[
 (b_{1,\nu},\ldots,b_{m,\nu})\in B_1\times\cdots\times B_m,
 \qquad
 \sum_{i=1}^m b_{i,\nu}=z_\nu.
\]
Furthermore,
\[
 N=1+\left\lceil\frac{z_+-z_-}{g}\right\rceil
 \le3+\frac{|U|}{Qg}
 \le3+\frac{3CX}{\Lambda}
 \le3+\frac{3\cdot2^{m+1}C}{\zeta}
 =O_{c,C,m,K,\zeta}(1).
\]
Finally,
\[
 \bigcup_{\nu=1}^N(F_i+b_{i,\nu}Q)
 \subseteq V_i,
 \qquad
 \left|\bigcup_{\nu=1}^N(F_i+b_{i,\nu}Q)\right|
 \le N|F_i|
 \ll M|R_i|.
\]
\end{proof}

\begin{lemma}\label{lem26}
Let $2\le m\le h$, and let the real intervals $\mathcal I_i=[a_i,b_i]$ satisfy
\[
 a_i<b_i\quad(1\le i\le m),\qquad
 \sum_{i=1}^ma_i<A\le B<\sum_{i=1}^mb_i.
\]
Put
\[
 D=\sum_{i=1}^m(b_i-a_i),\quad
 w=\min_i(b_i-a_i),\quad
 d=\min\left\{A-\sum_i a_i,\sum_i b_i-B\right\}.
\]
For any prescribed $0<c\le dw/(8D)$, one may take $C=3c$.
There are $N\le2+(B-A)/c$ real closed intervals $U_j$ and, for each $i$,
real closed intervals $V_{i,j}$ $(1\le j\le N)$ such that
\[
 [A,B]\subseteq\bigcup_j\operatorname{int}U_j,
 \qquad V_{i,j}\subseteq\operatorname{int}\mathcal I_i,
 \qquad U_j\subseteq\operatorname{int}\sum_{i=1}^mV_{i,j}.
\]
If $Xc\ge1$, put $U_j(X)=XU_j\cap\mathbb Z$ and
$V_{i,j}(X)=XV_{i,j}\cap\mathbb Z$. Then
\begin{equation}\label{eq222}
 \begin{gathered}
 \bigl[XA,XB\bigr]\cap\mathbb Z\subseteq\bigcup_jU_j(X),\\
 |V_{i,j}(X)|\ge cX,\qquad |U_j(X)|\le CX,\\
 U_j(X)+([-cX,cX]\cap\mathbb Z)
 \subseteq\sum_{i=1}^mV_{i,j}(X).
 \end{gathered}
\end{equation}
\end{lemma}

\begin{proof}
For each $t\in[A,B]$, put
\[
 \lambda_t=\frac{t-\sum_i a_i}{D}\in(0,1),
 \qquad x_i(t)=a_i+\lambda_t(b_i-a_i)\in(a_i,b_i).
\]
By definition,
\[
 \sum_{i=1}^m x_i(t)
 =\sum_i a_i+\lambda_t\sum_i(b_i-a_i)=t.
\]
Also,
\[
 \min\{\lambda_t,1-\lambda_t\}\ge d/D,\qquad
 \min_i\{x_i(t)-a_i,b_i-x_i(t)\}\ge dw/D\ge8c>4c.
\]
If $A=B$, take $N=1$ and $t_1=A$.
Otherwise, take
\[
 N=1+\left\lceil\frac{B-A}{c}\right\rceil,
 \qquad t_j=A+\frac{j-1}{N-1}(B-A)\quad(1\le j\le N).
\]
Then
\[
 N\le2+(B-A)/c,\qquad
 [A,B]\subseteq\bigcup_{j=1}^N(t_j-c,t_j+c).
\]
Put $x_{i,j}=x_i(t_j)$ and
\[
 U_j=[t_j-c,t_j+c],\qquad
 V_{i,j}=[x_{i,j}-4c,x_{i,j}+4c],\qquad C=3c.
\]
Since $\sum_i x_{i,j}=t_j$ and $m\ge2$,
\[
 V_{i,j}\subseteq\operatorname{int}\mathcal I_i,\qquad
 U_j\subseteq\operatorname{int}\sum_iV_{i,j}.
\]
If $Xc\ge1$, then
\[
 |V_{i,j}(X)|\ge8cX-1\ge cX,
 \qquad |U_j(X)|\le2cX+1\le CX.
\]
Since the sum of integer intervals is again an integer interval,
\begin{align*}
 \sum_i\min V_{i,j}(X)
 &\le X(t_j-4mc)+m
 \le X(t_j-2c)
 \le\min U_j(X)-cX,\\
 \sum_i\max V_{i,j}(X)
 &\ge X(t_j+4mc)-m
 \ge X(t_j+2c)
 \ge\max U_j(X)+cX.
\end{align*}
This proves \eqref{eq222}.
\end{proof}

\begin{lemma}\label{lem27}
Let $2\le d\le h$, let $j\le k$, put $q=q_j$ and $a=e_j$, and let
$S_0,\ldots,S_{d-1}\subseteq\mathbb Z_q\setminus\{a\}$ be nonempty. Then there exist
$R_i\subseteq[0,Q-1]\cap\mathbb Z$ $(0\le i<d)$ such that
\begin{align}
 x\bmod q&\in S_i\quad(x\in R_i),\qquad
 x\not\equiv e_\nu\pmod{q_\nu}
 \quad(x\in R_i,\ 1\le\nu\le k),\label{eq223}\\
 (R_0+\cdots+R_{d-1})\bmod Q
 &=\{u\bmod Q:u\bmod q\in S_0+\cdots+S_{d-1}\},\label{eq224}\\
 |R_i|&\ll_d |S_i|(Q/q)^{1/d}.\label{eq225}
\end{align}
If $d=2$ and $S_0=S_1=\{f\}$, one may also choose another pair of sets satisfying
\eqref{eq223} and \eqref{eq224} such that
\begin{equation}\label{eq226}
 |R_0|=2^{k-1},\qquad |R_1|\le Q/q.
\end{equation}
\end{lemma}

\begin{proof}
For $0\le i<d$, put
\[
 S_{j,i}=S_i.
\]
For $1\le\nu\le k$ with $\nu\ne j$, let
\[
 m_\nu=\lceil q_\nu^{1/d}\rceil.
\]
By Lemma~\ref{lem24}, with $q=q_\nu$ and $c=e_\nu$, choose
$\widetilde D_{\nu,i}\subseteq\mathbb Z_{q_\nu}\setminus\{e_\nu\}$ such that
\[
 |\widetilde D_{\nu,i}|=m_\nu,
 \qquad
 \sum_{i=0}^{d-1}\widetilde D_{\nu,i}=\mathbb Z_{q_\nu}.
\]
Define
\[
 S_{\nu,i}=
 \begin{cases}
 S_i,&\nu=j,\\
 \widetilde D_{\nu,i},&\nu\ne j,
 \end{cases}
 \qquad
 1\le\nu\le k,
 \quad0\le i<d,
\]
and
\[
 R_i=
 \left\{x\in[0,Q-1]\cap\mathbb Z:
 x\bmod q_\nu\in S_{\nu,i}
 \ (1\le\nu\le k)\right\}.
\]
Then
\[
 e_j=a\notin S_i=S_{j,i},
 \qquad
 e_\nu\notin S_{\nu,i}\quad(\nu\ne j),
\]
and therefore, for every $x\in R_i$ and $1\le\nu\le k$,
\[
 x\not\equiv e_\nu\pmod{q_\nu}.
\]
This proves \eqref{eq223}. By Lemma~\ref{lem23},
\begin{align*}
 (R_0+\cdots+R_{d-1})\bmod Q
 &=\left\{u\bmod Q:
 u\bmod q_\nu\in\sum_{i=0}^{d-1}S_{\nu,i}
 \ (1\le\nu\le k)\right\}\\
 &=\left\{u\bmod Q:
 u\bmod q\in S_0+\cdots+S_{d-1}\right\},
\end{align*}
which is \eqref{eq224}.

Again by Lemma~\ref{lem23},
\begin{align*}
 |R_i|
 &=|S_i|\prod_{\substack{\nu\le k\\\nu\ne j}}m_\nu\\
 &\le |S_i|\prod_{\substack{\nu\le k\\\nu\ne j}}
 q_\nu^{1/d}(1+q_\nu^{-1/d})\\
 &=|S_i|(Q/q)^{1/d}
 \prod_{\substack{\nu\le k\\\nu\ne j}}(1+q_\nu^{-1/d})\\
 &\le |S_i|(Q/q)^{1/d}
 \prod_{\nu\ge1}(1+q_\nu^{-1/h})\\
 &\ll_d |S_i|(Q/q)^{1/d}.
\end{align*}
Here $d\le h$ gives
\[
 q_\nu^{-1/d}\le q_\nu^{-1/h},
\]
and \eqref{eq25} gives
$\prod_{\nu\ge1}(1+q_\nu^{-1/h})<\infty$.
Thus \eqref{eq225} follows.

Suppose now that $d=2$ and $S_0=S_1=\{f\}$. For every $\nu\ne j$, choose
\[
 A_\nu=\{\alpha_\nu,\beta_\nu\}
 \subseteq\mathbb Z_{q_\nu}\setminus\{e_\nu\},
 \qquad
 \alpha_\nu\ne\beta_\nu,
\]
and let
\[
 B_\nu=\mathbb Z_{q_\nu}\setminus\{e_\nu\}.
\]
For every $s\in\mathbb Z_{q_\nu}$,
\[
 \#\{r\in A_\nu:s-r=e_\nu\}\le1<|A_\nu|=2.
\]
Hence we may choose $r\in A_\nu$ with $s-r\ne e_\nu$. Therefore
\[
 s=r+(s-r)\in A_\nu+B_\nu,
\]
and so
\[
 A_\nu+B_\nu=\mathbb Z_{q_\nu}.
\]
For the $j$th coordinate, put
\[
 A_j=B_j=\{f\},
 \qquad
 f\ne a=e_j.
\]
Define
\[
 R_0=\left\{x\in[0,Q-1]\cap\mathbb Z:x\bmod q_\nu\in A_\nu
 \ (1\le\nu\le k)\right\},
\]
\[
 R_1=\left\{x\in[0,Q-1]\cap\mathbb Z:x\bmod q_\nu\in B_\nu
 \ (1\le\nu\le k)\right\}.
\]
Then
\[
 (R_0+R_1)\bmod Q
 =\{u\bmod Q:u\bmod q\in\{2f\}\},
\]
\[
 |R_0|
 =|A_j|\prod_{\substack{\nu\le k\\\nu\ne j}}|A_\nu|
 =2^{k-1},
\]
and
\[
 |R_1|
 =|B_j|\prod_{\substack{\nu\le k\\\nu\ne j}}|B_\nu|
 =\prod_{\substack{\nu\le k\\\nu\ne j}}(q_\nu-1)
 \le\prod_{\substack{\nu\le k\\\nu\ne j}}q_\nu
 =Q/q.
\]
Thus \eqref{eq223}, \eqref{eq224}, and
\eqref{eq226} hold.
\end{proof}

\begin{lemma}\label{lem28}
Let $2\le m\le h$ be an integer.

\textup{(i)} Suppose that
\[
 \varnothing\ne R_i\subseteq[0,Q-1]\cap\mathbb Z\quad(0\le i<m),
\]
\[
 R_i\cap(e_\nu+q_\nu\mathbb Z)=\varnothing
 \quad(0\le i<m,\ 1\le\nu\le k),\qquad
 \Omega=(R_0+\cdots+R_{m-1})\bmod Q.
\]
Let $c,C,X>0$, and suppose that the integer intervals $U,V_0,\ldots,V_{m-1}$ satisfy
\[
 |V_i|\ge cX\quad(0\le i<m),\qquad |U|\le CX,\qquad
 U+[-cX,cX]\subseteq\sum_{i=0}^{m-1}V_i.
\]
Fix $0<\zeta\le c/8$, and take
\[
 0<\eta\le\min\left\{
 \frac{\zeta}{2^{m+3}(m+7)},
 \frac{c}{8\bigl(m(m+2)+4m+10\bigr)}\right\},
 \qquad Q\le\eta X.
\]
Then there exist sets $D_i\subseteq V_i$ $(0\le i<m)$ such that
\begin{align}
 U\cap\{n:n\bmod Q\in\Omega\}
 &\subseteq\sum_{i=0}^{m-1}D_i,\label{eq227}\\
 |D_i|&\ll(X/Q)^{1/m}|R_i|,\label{eq228}
\end{align}
and every $x\in D_i$ satisfies $x\bmod Q\in R_i$ and \eqref{eq24}.

\textup{(ii)} Let $u_i,v_i\in\mathbb Z$, and suppose that the finite sets
$F_i\subseteq\mathbb Z$ $(0\le i<m)$ satisfy
\[
 F_i\subseteq[u_i+(m+2)Q,v_i-(m+2)Q],\qquad
 F_i\cap(e_\nu+q_\nu\mathbb Z)=\varnothing
 \quad(1\le\nu\le k).
\]
Then there exist sets $F_i^*\subseteq[u_i,v_i]$ such that
\begin{equation}\label{eq229}
 \sum_{i=0}^{m-1}F_i\subseteq\sum_{i=0}^{m-1}F_i^*,\qquad
 |F_i^*|\le(2m+3)|F_i|.
\end{equation}
For every $x\in F_i^*$, there exists $z\in F_i$ such that $x\equiv z\pmod Q$, and $x$ satisfies \eqref{eq24}.
Moreover, for every $y\in\sum_{i=0}^{m-1}F_i^*$, there exists
$y_0\in\sum_{i=0}^{m-1}F_i$ such that
\[
 |y-y_0|\le m(m+1)Q.
\]

\textup{(iii)} Suppose that $m=2$, $n,\xi,\upsilon\in\mathbb Z$, and
\[
 |K_0\cap(n-K_1)|\ge3Q+1,\qquad
 \xi+\upsilon\equiv n\pmod Q,
\]
where $K_0,K_1$ are integer intervals and
\[
 \xi\not\equiv e_\nu\pmod{q_\nu},\qquad
 \upsilon\not\equiv e_\nu\pmod{q_\nu}
 \quad(1\le\nu\le k).
\]
Then there exists $(x,y)\in K_0\times K_1$ such that
\[
 x+y=n,\qquad x\equiv\xi\pmod Q,\qquad
 y\equiv\upsilon\pmod Q,
\]
and both $x$ and $y$ satisfy \eqref{eq24}.
\end{lemma}

\begin{proof}
\textup{(ii)} If $k=N$, or if $k<N$ and $q_{k+1}>2L$, put
\[
 F_i^*=F_i\subseteq[u_i,v_i],\qquad
 \sum_iF_i^*=\sum_iF_i,\qquad
 |F_i^*|=|F_i|,
\]
\[
 F_i^*\cap(e_\nu+q_\nu\mathbb Z)=\varnothing
 \quad(1\le\nu\le k).
\]
For every $y\in\sum_iF_i^*$, take
\[
 y_0=y\in\sum_iF_i,
 \qquad |y-y_0|=0.
\]
Now suppose that $k<N$ and $q_{k+1}\le2L$. Put
$\ell=q_{k+1}$ and $b=e_{k+1}$. Since $q_1,\ldots,q_{k+1}$ are distinct odd primes,
\[
 \ell>2,\qquad (Q,\ell)=1.
\]
By Lemma~\ref{lem22}, take
\[
 F_i^*=\{x+tQ:x\in F_i,\ t\in\mathbb Z,\ |t|\le m+1,\
 x+tQ\not\equiv b\pmod\ell\}.
\]
Then
\[
 \sum_iF_i\subseteq\sum_iF_i^*,\qquad
 |F_i^*|\le(2m+3)|F_i|.
\]
For $x+tQ\in F_i^*$,
\[
 x\in F_i\subseteq[u_i+(m+2)Q,v_i-(m+2)Q],
 \qquad |t|\le m+1,
\]
and hence
\[
 u_i+Q\le x+tQ\le v_i-Q,
 \qquad F_i^*\subseteq[u_i,v_i].
\]
Since $q_\nu\mid Q$ for $\nu\le k$,
\[
 \begin{aligned}
 x+tQ&\equiv x\pmod Q,\\
 x+tQ&\equiv x\not\equiv e_\nu\pmod{q_\nu}
 \quad(1\le\nu\le k),\\
 x+tQ&\not\equiv e_{k+1}\pmod{q_{k+1}}.
 \end{aligned}
\]
Thus the residue class modulo $Q$ is unchanged and \eqref{eq24} is satisfied.
Finally, if
\[
 y=\sum_{i=0}^{m-1}(x_i+t_iQ),\qquad
 x_i\in F_i,\quad |t_i|\le m+1,
\]
then, with $y_0=\sum_{i=0}^{m-1}x_i\in\sum_iF_i$,
\[
 |y-y_0|\le Q\sum_{i=0}^{m-1}|t_i|\le m(m+1)Q.
\]

\textup{(i)} Take $K=m+2$. Since
\[
 Q\le\eta X,
 \qquad
 \eta\le
 \min\left\{
 \frac{\zeta}{2^{m+3}(m+7)},
 \frac{c}{8\bigl(m(m+2)+4m+10\bigr)}
 \right\},
\]
Lemma~\ref{lem25} can be applied with $K=m+2$.
Put
\[
 M=\left\lfloor(\zeta X/Q)^{1/m}\right\rfloor\ge2,\qquad
 G_i=R_i+QM^i[0,M-1]\quad(0\le i<m).
\]
There are $N_0=O_{c,C,m,\zeta}(1)$ families of integers $b_{i,\mu}$ such that
\[
 H_i:=\bigcup_{\mu=1}^{N_0}(G_i+b_{i,\mu}Q)
 \subseteq[\min V_i+(m+2)Q,\max V_i-(m+2)Q],
\]
\[
 U\cap\{n:n\bmod Q\in\Omega\}
 \subseteq\sum_{i=0}^{m-1}H_i,
 \qquad |H_i|\le N_0M|R_i|.
\]
For every $x\in H_i$, its residue class modulo $Q$ lies in $R_i$. Since
$q_\nu\mid Q$ and
$R_i\cap(e_\nu+q_\nu\mathbb Z)=\varnothing$ for $1\le\nu\le k$,
we have $x\not\equiv e_\nu\pmod{q_\nu}$.
We may therefore apply \textup{(ii)} to
$H_0,\ldots,H_{m-1}$ and take $D_i=H_i^*$.
Then
\[
 \begin{aligned}
 D_i&\subseteq V_i,\\
 U\cap\{n:n\bmod Q\in\Omega\}
 &\subseteq\sum_iH_i\subseteq\sum_iD_i.
 \end{aligned}
\]
Moreover, every $x\in D_i$ satisfies $x\bmod Q\in R_i$ and \eqref{eq24}.
Since $Q\le\eta X$ and $\eta\le\zeta/2^m$,
\[
 (\zeta X/Q)^{1/m}\ge2.
\]
From $M=\lfloor(\zeta X/Q)^{1/m}\rfloor$ we get
\[
 \frac12(\zeta X/Q)^{1/m}
 \le M\le(\zeta X/Q)^{1/m}.
\]
Hence
\[
 |D_i|\le(2m+3)|H_i|\le(2m+3)N_0M|R_i|
 \ll(X/Q)^{1/m}|R_i|,
\]
which proves \eqref{eq227} and \eqref{eq228}.

\textup{(iii)} Let
\[
 I=K_0\cap(n-K_1),\qquad a=\min I,\qquad
 t\in[0,Q-1]\cap\mathbb Z,\quad t\equiv\xi-a\pmod Q.
\]
Since $I$ is an integer interval and $|I|\ge3Q+1$,
\[
 a+t+2Q\le a+3Q-1< a+3Q\le\max I.
\]
For $s=0,1,2$, put
\[
 x_s=a+t+sQ\in I,\qquad y_s=n-x_s\in K_1.
\]
Then
\[
 x_s\in K_0,\qquad x_s+y_s=n,\qquad
 x_s\equiv\xi\pmod Q,\qquad
 y_s\equiv n-\xi\equiv\upsilon\pmod Q,
\]
and
\[
 x_s\equiv\xi\not\equiv e_\nu\pmod{q_\nu},\qquad
 y_s\equiv\upsilon\not\equiv e_\nu\pmod{q_\nu}
 \quad(1\le\nu\le k).
\]
If $k=N$, or if $k<N$ and $q_{k+1}>2L$, take $s=0$.
Otherwise, put $\ell=q_{k+1}\le2L$. For $0\le r<s\le2$, since
$1\le s-r\le2<\ell$ and $(Q,\ell)=1$,
\[
 \begin{cases}
 x_s-x_r=(s-r)Q\not\equiv0\pmod\ell,\\
 y_s-y_r=-(s-r)Q\not\equiv0\pmod\ell.
 \end{cases}
\]
Thus
\[
 \begin{aligned}
 \#\{s\in\{0,1,2\}:x_s\equiv e_{k+1}\pmod\ell\}&\le1,\\
 \#\{s\in\{0,1,2\}:y_s\equiv e_{k+1}\pmod\ell\}&\le1,
 \end{aligned}
\]
and therefore
\[
 \#\{s\in\{0,1,2\}:x_s\equiv e_{k+1}\pmod\ell
 \ \text{or}\ y_s\equiv e_{k+1}\pmod\ell\}\le2<3.
\]
Hence there exists $s\in\{0,1,2\}$ such that
\[
 x_s\not\equiv e_{k+1}\pmod\ell,\qquad
 y_s\not\equiv e_{k+1}\pmod\ell.
\]
Taking $(x,y)=(x_s,y_s)$ proves the assertion.
\end{proof}

\begin{lemma}\label{lem29}
Let $Q,G,d\in\mathbb Z_{\ge1}$ and $b_*\in\mathbb Z$, and let $K\subseteq J$ be nonempty integer intervals such that
\[
 |K|>6G,\qquad
 I_m=[b_*+mG,b_*+(m+1)G],\qquad
 \mathcal M=\{m\in\mathbb Z:I_m\cap J\ne\varnothing\}.
\]
Let $\mathcal T\subseteq\mathbb Z$, and suppose that for every $m\in\mathcal M$,
\begin{equation}\label{eq230}
 I_m\cap\mathcal T\subseteq\mathcal F_m
 \subseteq[b_*+mG-dQ,b_*+(m+1)G+dQ].
\end{equation}
Then
\begin{equation}\label{eq231}
 |\mathcal M|\le2+|J|/G,
\end{equation}
and one may choose $m_0\in\mathcal M$ such that
\[
 A=b_*+m_0G,\qquad J_0=I_{m_0}=[A,A+G]\subseteq K,
 \qquad\operatorname{dist}(J_0,\partial K)\ge2G.
\]
Put
\[
 \mathcal F=\bigcup_{m\in\mathcal M\setminus\{m_0\}}\mathcal F_m.
\]
If $\Delta\in\mathbb Z_{\ge0}$ and
\[
 \mathcal F\subseteq\mathcal H
 \subseteq\mathcal F+[-\Delta,\Delta],
\]
then, with
\begin{equation}\label{eq232}
 J_1=[A+(d+1)Q+\Delta,A+G-(d+1)Q-\Delta],
\end{equation}
we have
\begin{equation}\label{eq233}
 J_1\subseteq J_0\subseteq K,\qquad
 (J\setminus J_0)\cap\mathcal T\subseteq\mathcal H,
 \qquad J_1\cap\mathcal H=\varnothing.
\end{equation}
If $G>2(d+1)Q+2\Delta$, then
\begin{equation}\label{eq234}
 |J_0|=G+1,\qquad
 |J_1|=G-2(d+1)Q-2\Delta+1.
\end{equation}
\end{lemma}

\begin{proof}
Write $J=[u,v]$ and $K=[u',v']$, where all endpoints are integers. If $I_m\cap J\ne\varnothing$, then
\[
 u-G\le b_*+mG\le v,
\]
and hence
\[
 |\mathcal M|\le\frac{v-u+G}{G}+1
 \le2+\frac{|J|}{G}.
\]
Since $G\ge1$, let
\[
 m_0=\left\lceil\frac{u'+2G-b_*}{G}\right\rceil,
 \qquad A=b_*+m_0G.
\]
By the definition of the ceiling function,
\[
 u'+2G\le A<u'+3G.
\]
Since $|K|=v'-u'+1>6G$,
\[
 A-u'\ge2G,\qquad
 v'-(A+G)\ge u'+6G-(u'+4G-1)>2G.
\]
Therefore $J_0=[A,A+G]\subseteq K\subseteq J$, with
$m_0\in\mathcal M$ and $\operatorname{dist}(J_0,\partial K)\ge2G$.

Since $\min I_{m+1}=\max I_m$, we have $\bigcup_{m\in\mathbb Z}I_m=\mathbb Z$.
Let $n\in(J\setminus J_0)\cap\mathcal T$, and choose $m$ such that $n\in I_m$. Then
\[
 m\in\mathcal M\setminus\{m_0\},\qquad
 n\in I_m\cap\mathcal T\subseteq\mathcal F_m\subseteq\mathcal F.
\]
Thus $(J\setminus J_0)\cap\mathcal T\subseteq\mathcal F\subseteq\mathcal H$.
If $m\le m_0-1$, then by \eqref{eq230},
\[
 \mathcal F_m\subseteq(-\infty,b_*+(m+1)G+dQ]
 \subseteq(-\infty,b_*+m_0G+dQ]=(-\infty,A+dQ].
\]
If $m\ge m_0+1$, similarly,
\[
 \mathcal F_m\subseteq[b_*+mG-dQ,\infty)
 \subseteq[b_*+(m_0+1)G-dQ,\infty)=[A+G-dQ,\infty).
\]
Hence, since $(d+1)Q>dQ$,
\[
 [A+(d+1)Q,A+G-(d+1)Q]\cap\mathcal F=\varnothing.
\]
If $y\in J_1\cap\mathcal H$, then $\mathcal H\subseteq\mathcal F+[-\Delta,\Delta]$ gives an element
$y_0\in\mathcal F$ with $|y-y_0|\le\Delta$. Thus
\[
 y_0\in[y-\Delta,y+\Delta]\cap\mathcal F
 \subseteq[A+(d+1)Q,A+G-(d+1)Q]\cap\mathcal F
 =\varnothing,
\]
a contradiction. Hence \eqref{eq233} holds.
If $G>2(d+1)Q+2\Delta$, then $J_1$ is nonempty and, since its endpoints are integers,
\[
 |J_1|=A+G-(d+1)Q-\Delta
       -\bigl(A+(d+1)Q+\Delta\bigr)+1
       =G-2(d+1)Q-2\Delta+1.
\]
Together with $|J_0|=G+1$, this proves \eqref{eq234}.
\end{proof}

\section{Main lemmas and proof of Theorem~\ref{thm11}}\label{sec3}

Fix an integer $b>4h$, and put
\[
 S_b=1+b+\cdots+b^{p-1},\qquad M_b=10hS_b.
\]
Define
\begin{equation}\label{eq31}
 H=hM_b,\quad O_i=pM_b+b^i\ (0\le i<p),\quad
 C=pM_b-S_b,\quad Z=hpM_b.
\end{equation}

\begin{lemma}\label{lem31}
The numbers defined in \eqref{eq31} satisfy
\[
 0<C<O_i<H\quad(0\le i<p),\qquad
 O_i\ne O_j\quad(i\ne j).
\]
If
\[
 d_1,\ldots,d_h\in\{0,H,C,O_0,\ldots,O_{p-1}\},
 \qquad
 \sum_{\nu=1}^h d_\nu=Z,
\]
then the corresponding multiset must be one of the following two multisets:
\begin{equation}\label{eq32}
 \{0,\underbrace{H,\ldots,H}_{p}\},
 \qquad
 \{C,O_0,\ldots,O_{p-1}\}.
\end{equation}
Conversely, each of the two multisets in \eqref{eq32} has sum $Z$.
\end{lemma}

\begin{proof}
Since $p=h-1$, $M_b=10hS_b$, and $b^i\le S_b$,
\begin{align*}
 C&=pM_b-S_b>0, & O_i-C&=S_b+b^i>0,\\
 H-O_i&=M_b-b^i>0, & O_i-O_j&=b^i-b^j\ne0\quad(i\ne j).
\end{align*}
Let
\[
 x=\#\{\nu:d_\nu=H\},
 \quad
 z=\#\{\nu:d_\nu=C\},
 \quad
 n_i=\#\{\nu:d_\nu=O_i\},
 \quad
 y=z+\sum_{i=0}^{p-1}n_i.
\]
Then
\[
 x+y\le h.
\]
Using $\sum_\nu d_\nu=Z$ and \eqref{eq31}, we obtain
\begin{align*}
 0
 &=xH+zC+\sum_{i=0}^{p-1}n_iO_i-Z\\
 &=hxM_b+z(pM_b-S_b)
   +\sum_{i=0}^{p-1}n_i(pM_b+b^i)-hpM_b\\
 &=(hx+py-hp)M_b+\sum_{i=0}^{p-1}(n_i-z)b^i.
\end{align*}
Moreover,
\[
 \left|\sum_{i=0}^{p-1}(n_i-z)b^i\right|
 \le\sum_{i=0}^{p-1}|n_i-z|b^i
 \le hS_b
 =\frac1{10}M_b
 <\frac12M_b.
\]
Thus
\[
 (hx+py-hp)M_b\in M_b\mathbb Z,
 \qquad
 \left|(hx+py-hp)M_b\right|<\frac12M_b,
\]
and hence
\[
 hx+py=hp.
\]

First suppose that $p\ge2$. Since $h=p+1$,
\[
 hx+py=(p+1)x+py=hp.
\]
Reducing modulo $p$ gives
\[
 x\equiv0\pmod p.
\]
Since
\[
 0\le x\le h=p+1,
 \qquad
 x\in\{0,p\}.
\]
If $x=p$, then $y=0$; if $x=0$, then $y=h$.
If $p=1$, the equality $hx+py=hp$ similarly gives
\[
 2x+y=2,
 \qquad x+y\le2,
 \qquad x,y\ge0.
\]
Hence
\[
 (x,y)\in\{(1,0),(0,2)\}.
\]
When $(x,y)=(p,0)$,
\[
 \#\{\nu:d_\nu=0\}=h-p=1,
\]
and therefore
\[
 \{d_1,\ldots,d_h\}=\{0,H,\ldots,H\}.
\]
On the other hand, when $(x,y)=(0,h)$, the original equality becomes
\[
 \sum_{i=0}^{p-1}(n_i-z)b^i=0.
\]
Also,
\[
 |n_i-z|\le h<b/2.
\]
For $0\le r<p$, suppose that
\[
 n_0=\cdots=n_{r-1}=z
\]
(with no condition when $r=0$). Then
\[
 \sum_{i=0}^{p-1}(n_i-z)b^i=0
\]
gives
\[
 \sum_{i=r}^{p-1}(n_i-z)b^{i-r}=0,
\]
and therefore
\[
 n_r-z\equiv0\pmod b,
 \qquad |n_r-z|\le h<b,
 \qquad n_r=z.
\]
By induction,
\[
 n_i=z\qquad(0\le i<p).
\]
Since
\[
 z+\sum_{i=0}^{p-1}n_i=(p+1)z=h=p+1,
\]
we have $z=1$. Hence
\[
 n_0=\cdots=n_{p-1}=1.
\]
Finally,
\[
 pH=phM_b=Z,
\qquad
 C+\sum_{i=0}^{p-1}O_i
 =(pM_b-S_b)+p^2M_b+S_b
 =p(p+1)M_b=Z.
\]
\end{proof}

Put
\[
 \omega=\frac14,\qquad \alpha=\frac{\omega}{Z},
 \qquad
 \mathcal D:=\{H,C,O_0,\ldots,O_{p-1}\}.
\]
Fix
\begin{equation}\label{eq33}
 0<\rho<\min\left\{\frac{\alpha}{100h},\frac{\omega}{100h^2p}\right\},
 \qquad \tau=2\rho,\qquad \eps=\rho/100.
\end{equation}
For $d\in\mathcal D$, put
\begin{equation}\label{eq34}
 B_d(L):=[\alpha dL-\rho L,\alpha dL+\rho L].
\end{equation}

\begin{lemma}\label{lem32}
For every $L>0$,
\begin{equation}\label{eq35}
 [0,\tau L]\cap B_d(L)=\varnothing,
 \qquad
 B_d(L)\cap B_{d'}(L)=\varnothing
 \quad(d,d'\in\mathcal D,\ d\ne d').
\end{equation}
Suppose that
\[
 d_1,\ldots,d_h\in\{0\}\cup\mathcal D,
 \qquad
 x_i\in
 \begin{cases}
 [0,\tau L],&d_i=0,\\
 B_{d_i}(L),&d_i\in\mathcal D,
 \end{cases}
 \qquad
 n=\sum_{i=1}^h x_i,
\]
and
\[
 n\in[(\omega-\eps)L,(\omega+\eps)L].
\]
Then
\begin{equation}\label{eq36}
 \begin{aligned}
 \sum_{i=1}^h d_i&=Z,\\
 \{d_1,\ldots,d_h\}
 &\in\left\{
 \{0,\underbrace{H,\ldots,H}_{p}\},
 \{C,O_0,\ldots,O_{p-1}\}
 \right\}.
 \end{aligned}
\end{equation}
\end{lemma}

\begin{proof}
By \eqref{eq33}, $\tau=2\rho$, and the fact that $d,d'$ are distinct positive integers,
\[
 d\ge1,\quad |d-d'|\ge1,
 \qquad \rho<\frac\alpha{100h}<\frac\alpha3.
\]
Hence
\[
 \alpha d-\rho\ge\alpha-\rho>2\rho=\tau,
 \qquad \alpha|d-d'|\ge\alpha>2\rho,
\]
which proves \eqref{eq35}.
Moreover, from the intervals containing the $x_i$,
\[
 \left|\frac{x_i}{L}-\alpha d_i\right|
 \le
 \begin{cases}
 \tau,&d_i=0,\\
 \rho,&d_i\in\mathcal D.
 \end{cases}
\]
Since $n=\sum_i x_i$, $\omega=\alpha Z$, and $|n/L-\omega|\le\eps$,
\begin{align*}
 \left|\alpha\left(\sum_{i=1}^h d_i-Z\right)\right|
 &=\left|\sum_{i=1}^h\left(\alpha d_i-\frac{x_i}{L}\right)
   +\left(\frac nL-\omega\right)\right|\\
 &\le h\tau+h\rho+\eps\\
 &=\left(3h+\frac1{100}\right)\rho
 <\frac{3h+1/100}{100h}\alpha
 <\frac\alpha2.
\end{align*}
Since $\sum_i d_i-Z\in\mathbb Z$,
\[
 \left|\sum_{i=1}^h d_i-Z\right|<\frac12,
 \qquad \sum_{i=1}^h d_i=Z.
\]
The second assertion in \eqref{eq36} now follows from
\eqref{eq32} in Lemma~\ref{lem31}.
\end{proof}

Put $\mu=3h/4$ and $W=\omega+\eps$. Fix
\[
 R>\max\left\{2,\frac1\tau,\frac{5h}{\eps},
                  \frac{5h}{8(\omega-\eps)}\right\},
\]
and let
\begin{equation}\label{eq37}
 \delta=\frac1{2R},\qquad\lambda=\frac{5h}{8R}.
\end{equation}
Then $1/R<\tau$, $\delta<\eps/(10h)$, and $\lambda<\omega-\eps$, and
\begin{equation}\label{eq38}
 h\delta<\lambda,\qquad \lambda R<\mu,\qquad \delta R=1/2>W.
\end{equation}
Indeed, by \eqref{eq37},
\[
 h\delta=\frac h{2R}<\frac{5h}{8R}=\lambda,
 \qquad \lambda R=\frac{5h}{8}<\frac{3h}{4}=\mu.
\]
Also, since $\omega=1/4$ and $\eps=\rho/100<1/4$,
\[
 \delta R=\frac12>\frac14+\eps=W.
\]
Put
\[
 u=\frac{\omega-\eps/2}{h},\qquad v=\omega+\eps/2,
 \qquad w_0=\min\{u-\delta,1-v\},
\]
\[
 d_0=\frac18\min\{\lambda-h\delta,\eps/2-p\delta,
                         h-\mu,1-2v-h\delta\},
\]
\[
 c_{\rm E}=\frac{d_0w_0}{16h},\qquad
 c_{\rm J}=\frac14\min\{p\rho-\delta-\eps,
                     (p+1)\rho-\eps,\rho-\delta,2\rho\},
\]
\[
 c_{\rm g}=\min\{c_{\rm E},c_{\rm J},1/(10p)\}>0.
\]
All the quantities in these minima are positive by \eqref{eq33},
\eqref{eq37}, and \eqref{eq38}. They are fixed before the parameters below.
Fix
\begin{equation}\label{eq39}
 \begin{gathered}
 0<\sigma<\frac{\eps}{100h},\\
 0<\zeta\le
 \min\left\{\frac12,\frac{\eps}{24\sigma},\frac{c_{\rm g}}8\right\},\\
 0<\eta\le
 \min\left\{
 \frac1{2\sigma},\frac{\zeta}{2^p},
 \frac{2^{-p-2}\zeta}{p+1},\frac{\zeta}{1000}
 \right\},\\
 \eta\le\min\left\{
 \frac{2^{-p-1}\zeta}{(2h+3)(h+1)},
 \frac{\rho}{8(2h+3)\sigma}
 \right\},\\
 \eta\le\min\left\{
 \frac{\zeta}{2^{h+3}(h+7)},
 \frac{c_{\rm g}}{8\bigl(h(h+2)+4h+10\bigr)}
 \right\},\\
 q_0>\max\left\{1000h^3,\frac{96}{\zeta}\right\},\qquad
 C_*>\max\left\{1,\frac{96}{\zeta\eta q_0^4}\right\}.
 \end{gathered}
\end{equation}
At every application of Lemma~\ref{lem28}\textup{(i)} below,
$2\le m\le h$ and $c\ge c_{\rm g}$. Therefore \eqref{eq39} gives
\[
 0<\zeta\le\frac c8,
 \qquad
 0<\eta\le
 \min\left\{
 \frac{\zeta}{2^{m+3}(m+7)},
 \frac{c}{8\bigl(m(m+2)+4m+10\bigr)}
 \right\}.
\]
These inequalities also give the required bounds in Lemma~\ref{lem25}
when it is used in the proof of Lemma~\ref{lem28}, with $K=m+2$.
The three margins used below are $c_{\rm E}$, $c_{\rm J}$ and, when $p\ge2$,
$1/(10p)$. Thus no subsequent reduction of $\zeta$ or $\eta$ is needed.
All implied constants below may depend on $h,\theta$ and the fixed parameters,
but are uniform in $L,N,j$ and the primes unless otherwise indicated.
Let $L$ be sufficiently large in terms of the fixed parameters, and suppose that
\begin{equation}\label{eq310}
 A_0=\{e_1,\ldots,e_N\}\subseteq[0,L/R],\qquad
 T=\lfloor\sigma L^\kappa\rfloor,\qquad Q_1\le\eta T.
\end{equation}
Assume that $q_1,\ldots,q_N$ satisfy \eqref{eq21} and that
$A_0$ satisfies \eqref{eq22}. Let $k,Q$ be given by \eqref{eq23}. Then
\[
 1\le k\le N,\qquad Q=Q_k\le\eta T.
\]

Fix
\[
 a:=e_j,\qquad q:=q_j,
\]
and require
\begin{equation}\label{eq311}
 j\le k,\qquad a\le\rho T/10,\qquad q_j^p\le L^{h\theta-1}.
\end{equation}
Under \eqref{eq310} and \eqref{eq311}, put
\begin{equation}\label{eq312}
 p:=h-1,\qquad B_0:=100h.
\end{equation}
Let
\begin{equation}\label{eq313}
 r\equiv a+hB_0\pmod q,
 \qquad
 t\equiv a+pr\pmod q.
\end{equation}
Define
\begin{equation}\label{eq314}
 \mathcal R_q
 :=\{u\bmod q:u\in\mathbb Z,\ pB_0-2\le u\le hB_0\}
 \subseteq\mathbb Z_q.
\end{equation}
Also put
\begin{equation}\label{eq315}
 m_q:=\left\lceil\frac qh\right\rceil,
 \qquad
 \mathcal B_q
 :=\{u\bmod q:u\in\mathbb Z,\ hB_0\le u\le hB_0+m_q\}.
\end{equation}

\begin{lemma}\label{lem33}
If $q\ge q_0>10h^2B_0$, then
\begin{equation}\label{eq316}
 t\equiv ha+phB_0\pmod q.
\end{equation}
Moreover,
\[
 \{pB_0,hB_0\}\bmod q\subseteq\mathcal R_q.
\]
When $h=2$,
\[
 \bigl(\{B_0+u:-2\le u\le2\}\cup\{2B_0\}\bigr)\bmod q
 \subseteq\mathcal R_q.
\]
Furthermore,
\begin{equation}\label{eq317}
 0<hB_0<hB_0+m_q<q,
 \qquad 0\notin\mathcal B_q,
\end{equation}
and
\begin{equation}\label{eq318}
 h\mathcal B_q
 =[h^2B_0,h^2B_0+hm_q]\pmod q
 =\mathbb Z_q.
\end{equation}
If
\[
 1\le s\le p,\quad
 \beta_i\in[hB_0,hB_0+m_q]\ (1\le i\le s),\quad
 \gamma_j\in[pB_0-2,hB_0]\ (1\le j\le h-s),
\]
and
\[
 D:=\sum_{i=1}^s\beta_i+
 \sum_{j=1}^{h-s}\gamma_j-phB_0,
\]
then
\begin{equation}\label{eq319}
 0<D<q,\qquad
 ha+\sum_{i=1}^s\beta_i+
 \sum_{j=1}^{h-s}\gamma_j
 \equiv t+D\not\equiv t\pmod q.
\end{equation}
\end{lemma}

\begin{proof}
Since $p=h-1$, \eqref{eq313} gives
\[
 t\equiv a+p(a+hB_0)=ha+phB_0\pmod q.
\]
By \eqref{eq314},
\[
 \{pB_0,hB_0\}\subseteq[pB_0-2,hB_0].
\]
When $h=2$,
\[
 \{B_0-2,B_0-1,B_0,B_0+1,B_0+2,2B_0\}
 \subseteq[B_0-2,2B_0].
\]
Since $m_q\le q/h+1$ and $q>10h^2B_0$,
\[
 q-(hB_0+m_q)
 \ge\left(1-\frac1h\right)q-hB_0-1>0,
\]
which proves \eqref{eq317}. Since $hm_q\ge q$,
\[
 h\mathcal B_q
 =[h^2B_0,h^2B_0+hm_q]\pmod q
 =\mathbb Z_q,
\]
which is \eqref{eq318}.

Finally, since $B_0=100h$ and $1\le s\le p=h-1$,
\begin{align*}
 D
 &\ge shB_0+(h-s)(pB_0-2)-phB_0\\
 &=sB_0-2(h-s)\\
 &\ge B_0-2(h-1)>0,
\end{align*}
and
\begin{align*}
 D
 &\le s(hB_0+m_q)+(h-s)hB_0-phB_0\\
 &=hB_0+sm_q\\
 &\le hB_0+p\left(\frac qh+1\right)\\
 &=q-\frac qh+hB_0+p<q.
\end{align*}
Together with \eqref{eq316}, this gives
\[
 ha+\sum_i\beta_i+\sum_j\gamma_j
 \equiv ha+phB_0+D
 \equiv t+D\not\equiv t\pmod q.
\]
\end{proof}

\begin{lemma}\label{lem34}
Assume that \eqref{eq310}, \eqref{eq21}, \eqref{eq22}, and \eqref{eq311} hold. Then there exists a set
\[
 \mathcal E(L)\subseteq[\delta L,L]\cap\mathbb Z
\]
such that
\begin{equation}\label{eq320}
 |\mathcal E(L)|=O(L^\theta),
\end{equation}
and
\begin{equation}\label{eq321}
\begin{aligned}
 x&\not\equiv e_i\pmod{q_i}
 &&(x\in\mathcal E(L),\ 1\le i\le k),\\
 x&\not\equiv e_{k+1}\pmod{q_{k+1}}
 &&(x\in\mathcal E(L),\ k<N,\ q_{k+1}\le2L).
\end{aligned}
\end{equation}
Let
\[
 J:=[(\omega-\eps)L,(\omega+\eps)L].
\]
Then
\begin{equation}\label{eq322}
 [\lambda L,\mu L]\setminus J\subseteq h\mathcal E(L).
\end{equation}
Furthermore, suppose that
\begin{equation}\label{eq323}
 \mathcal C\subseteq
 \left(
 [0,\tau L]\cup
 \bigcup_{d\in\mathcal D}B_d(L)
 \right)\cap\mathbb Z,
\end{equation}
and that for every $x\in\mathcal C\cap\bigcup_{d\in\mathcal D}B_d(L)$,
\begin{equation}\label{eq324}
 (x-a)\bmod q\in\mathcal R_q.
\end{equation}
Then, for
\[
 K:=[(\omega-\eps/4)L,(\omega+\eps/4)L],
\]
we have
\begin{equation}\label{eq325}
 \left(K\cap(t+q\mathbb Z)\right)
 \cap
 \bigcup_{s=1}^h
 \left(s\mathcal E(L)+(h-s)\mathcal C\right)
 =\varnothing.
\end{equation}
\end{lemma}

\begin{proof}
Let
\[
 J_0:=[\omega-\eps,\omega+\eps],
 \qquad J=(LJ_0)\cap\mathbb Z,
\]
and let $I_-,I_+,I_r$ be the following real intervals:
\[
 u:=\frac{\omega-\eps/2}{h},
 \qquad
 v:=\omega+\frac\eps2,
 \qquad
 I_-:=[\delta,u],
 \qquad
 I_+:=[v,1],
\]
\begin{equation}\label{eq326}
 I_r:=rI_++(h-r)I_-
 =[rv+(h-r)\delta,\ r+(h-r)u]
 \qquad(0\le r\le h).
\end{equation}
By \eqref{eq38},
\begin{align*}
 \min I_0&=h\delta<\lambda,\\
 \max I_0&=hu=\omega-\frac\eps2>\omega-\eps,
\end{align*}
and hence
\begin{equation}\label{eq327}
 [\lambda,\omega-\eps]\subseteq\operatorname{int}I_0.
\end{equation}
Also,
\begin{align*}
 \min I_1
 &=v+(h-1)\delta
 <\omega+\frac\eps2+\frac\eps{10}
 <\omega+\eps,\\
 \max I_1
 &=1+(h-1)u
 >\omega+\eps,
\end{align*}
and therefore
\begin{equation}\label{eq328}
 \omega+\eps\in\operatorname{int}I_1.
\end{equation}
For $1\le r<h$,
\begin{align*}
 \max I_r-\min I_{r+1}
 &=r-(r+1)v+(h-r)u-(h-r-1)\delta\\
 &\ge r-(r+1)v-(h-r-1)\delta\\
 &\ge1-2v-h\delta\\
 &=1-2\omega-\eps-h\delta>0.
\end{align*}
Thus
\begin{equation}\label{eq329}
 I_r\cap I_{r+1}\ne\varnothing
 \qquad(1\le r<h).
\end{equation}
Since
\[
 \max I_h=h>\mu,
\]
we obtain
\begin{equation}\label{eq330}
 [\lambda,\omega-\eps]\cup[\omega+\eps,\mu]
 \subseteq\bigcup_{r=0}^h\operatorname{int}I_r.
\end{equation}
Put
\[
 \mathcal A:=[\lambda,\omega-\eps]\cup[\omega+\eps,\mu].
\]
By the definition of $d_0$ and the preceding endpoint estimates,
\[
 \begin{gathered}
 \lambda-\min I_0\ge8d_0,\qquad
 \max I_0-(\omega-\eps)\ge8d_0,\\
 (\omega+\eps)-\min I_1\ge8d_0,\qquad
 \max I_h-\mu\ge8d_0,\\
 \max I_r-\min I_{r+1}\ge8d_0\quad(1\le r<h).
 \end{gathered}
\]
Define the real closed intervals
\[
 W_0=[\lambda,\omega-\eps],\qquad
 W_r=[\omega+\eps,\mu]\cap
       [\min I_r+d_0,\max I_r-d_0]\quad(1\le r\le h),
\]
discarding empty $W_r$. Their union is $\mathcal A$, and each $W_r$ lies
at distance at least $d_0$ from both endpoints of $I_r$.
For each remaining $r$, apply Lemma~\ref{lem26} to $[A,B]=W_r$ and
\[
 \mathcal I_\nu=
 \begin{cases}
 I_-,&1\le\nu\le h-r,\\
 I_+,&h-r<\nu\le h,
 \end{cases}
 \qquad(1\le\nu\le h).
\]
In the notation of that lemma, $d\ge d_0$, $w\ge w_0$ and $D\le h$, so
\[
 \frac{dw}{8D}\ge\frac{d_0w_0}{8h}=2c_{\rm E}.
\]
Thus the same choice $c=c_{\rm E}$, $C=3c_{\rm E}$ is valid for every $r$.
Combining and relabeling the finitely many resulting intervals, we obtain $N_0$ families of real closed intervals
$U_m,V_{m,\nu}$ $(1\le m\le N_0$, $1\le\nu\le h)$, where
\[
 N_0\in\mathbb Z_{\ge1},\qquad
 N_0\le(h+1)(2+h/c_{\rm E})=O(1),
 \qquad
 r_m\in\{0,\ldots,h\},
\]
such that
\begin{align}
 &\prod_{\nu=1}^hV_{m,\nu}
 \subseteq(\operatorname{int}I_-)^{h-r_m}\times(\operatorname{int}I_+)^{r_m},
 \label{eq331}\\
 &\mathcal A\subseteq\bigcup_{m=1}^{N_0}U_m,
 \label{eq332}\\
 &U_m\subseteq\sum_{\nu=1}^hV_{m,\nu}
 \qquad(1\le m\le N_0).
 \label{eq333}
\end{align}
By \eqref{eq222} in Lemma~\ref{lem26}, for $L\ge1/c_{\rm E}$,
\begin{equation}\label{eq334}
 \begin{gathered}
 |LV_{m,\nu}\cap\mathbb Z|\ge c_{\rm E}L,\qquad
 |LU_m\cap\mathbb Z|\le3c_{\rm E}L,\\
 (LU_m\cap\mathbb Z)+([-c_{\rm E}L,c_{\rm E}L]\cap\mathbb Z)
 \subseteq\sum_{\nu=1}^h(LV_{m,\nu}\cap\mathbb Z).
 \end{gathered}
\end{equation}

Fix $m$. By \eqref{eq317} and \eqref{eq318},
\[
 a\notin a+\mathcal B_q\subseteq\mathbb Z_q,
 \qquad h(a+\mathcal B_q)=\mathbb Z_q.
\]
Apply Lemma~\ref{lem27} with
\[
 d=h,\qquad S_{\nu-1}=a+\mathcal B_q
 \quad(1\le\nu\le h),
\]
and denote the resulting set $R_{\nu-1}$ by
$R_{m,\nu}\subseteq[0,Q-1]\cap\mathbb Z$.
By \eqref{eq223},
\begin{equation}\label{eq335}
 x\bmod q\in a+\mathcal B_q,\qquad
 x\not\equiv e_i\pmod{q_i}
 \quad(x\in R_{m,\nu},\ 1\le i\le k,\ 1\le\nu\le h).
\end{equation}
By \eqref{eq224} and $h(a+\mathcal B_q)=\mathbb Z_q$,
\begin{equation}\label{eq336}
 (R_{m,1}+\cdots+R_{m,h})\bmod Q
 =\{u\bmod Q:u\bmod q\in h(a+\mathcal B_q)\}
 =\mathbb Z_Q.
\end{equation}
Finally, by \eqref{eq225} and $|\mathcal B_q|\le q$,
\begin{equation}\label{eq337}
 |R_{m,\nu}|\ll_h|\mathcal B_q|\left(\frac Qq\right)^{1/h}
 \ll_h q\left(\frac Qq\right)^{1/h}.
\end{equation}

For each $m$, \eqref{eq336} gives $\Omega=\mathbb Z_Q$.
Apply Lemma~\ref{lem28}\textup{(i)} with $h$ summands and
\[
 X=L,\quad U=LU_m\cap\mathbb Z,\quad
 V_\nu=LV_{m,\nu}\cap\mathbb Z,\quad
 R_\nu=R_{m,\nu}\quad(1\le\nu\le h).
\]
By \eqref{eq334}, \eqref{eq335}, and
\[
 Q\le\eta T\le\eta\sigma L,
\]
all the hypotheses of Lemma~\ref{lem28}\textup{(i)} are satisfied.
Denote the resulting sets by $F_{m,\nu}$. Then
\begin{equation}\label{eq338}
 (LU_m)\cap\mathbb Z\subseteq F_{m,1}+\cdots+F_{m,h},
\end{equation}
\begin{equation}\label{eq339}
 F_{m,\nu}\subseteq LV_{m,\nu},\qquad
 |F_{m,\nu}|\ll(L/Q)^{1/h}|R_{m,\nu}|,
\end{equation}
and every element of every $F_{m,\nu}$ satisfies \eqref{eq24}.
Let
\[
 \mathcal E(L):=
 \bigcup_{m=1}^{N_0}\bigcup_{\nu=1}^hF_{m,\nu}.
\]
By \eqref{eq331} and \eqref{eq339},
\begin{equation}\label{eq340}
 \mathcal E(L)
 \subseteq LI_-\cup LI_+
 \subseteq[\delta L,L].
\end{equation}
Let $n\in[\lambda L,\mu L]\setminus J$. By \eqref{eq332}, there exists
$m\in\{1,\ldots,N_0\}$ such that $n/L\in U_m$. Hence \eqref{eq338} gives
\[
 n\in F_{m,1}+\cdots+F_{m,h}\subseteq h\mathcal E(L).
\]
Thus \eqref{eq322} follows.

By \eqref{eq339} and \eqref{eq337},
\begin{align*}
 |F_{m,\nu}|
 &\ll\left(\frac LQ\right)^{1/h}
 q\left(\frac Qq\right)^{1/h}\\
 &=q^{1-1/h}L^{1/h}\\
 &=q^{p/h}L^{1/h}\\
 &=(q^p)^{1/h}L^{1/h}\\
 &\le L^{(h\theta-1)/h}L^{1/h}\\
 &=L^\theta.
\end{align*}
Therefore
\[
 |\mathcal E(L)|
 \le\sum_{m=1}^{N_0}\sum_{\nu=1}^h|F_{m,\nu}|
 =O(L^\theta),
\]
which is \eqref{eq320}. Also,
\begin{align*}
 x&\not\equiv e_i\pmod{q_i}
 &&(x\in\mathcal E(L),\ 1\le i\le k),\\
 x&\not\equiv e_{k+1}\pmod{q_{k+1}}
 &&(x\in\mathcal E(L),\ k<N,\ q_{k+1}\le2L),
\end{align*}
which is \eqref{eq321}.
Moreover,
\[
 x\in F_{m,\nu}
 \Longrightarrow x\bmod Q\in R_{m,\nu}
 \Longrightarrow x\bmod q\in a+\mathcal B_q,
\]
and hence
\begin{equation}\label{eq341}
 (x-a)\bmod q\in\mathcal B_q\qquad(x\in\mathcal E(L)).
\end{equation}

To prove \eqref{eq325}, suppose to the contrary that there exist
\[
 1\le s\le h,
 \qquad
 x_1,\ldots,x_s\in\mathcal E(L),
 \qquad
 y_1,\ldots,y_{h-s}\in\mathcal C,
\]
such that
\begin{equation}\label{eq342}
 n:=\sum_{i=1}^s x_i+
 \sum_{j=1}^{h-s}y_j
 \in K\cap(t+q\mathbb Z).
\end{equation}
By \eqref{eq340},
\[
 x_i\in LI_-\cup LI_+
 \qquad(1\le i\le s).
\]
If $x_i\in LI_+$ for some $i$, then
\[
 n\ge x_i\ge vL
 =\left(\omega+\frac\eps2\right)L
 >\left(\omega+\frac\eps4\right)L
 \ge\max K,
\]
contradicting \eqref{eq342}. Hence
\begin{equation}\label{eq343}
 x_i\in LI_-
 \qquad(1\le i\le s).
\end{equation}
If $s=h$, then by \eqref{eq343},
\[
 n=\sum_{i=1}^h x_i
 \le huL
 =\left(\omega-\frac\eps2\right)L
 <\left(\omega-\frac\eps4\right)L
 \le\min K,
\]
again contradicting \eqref{eq342}. Therefore
\begin{equation}\label{eq344}
 1\le s\le h-1=p.
\end{equation}

By \eqref{eq323} and \eqref{eq344}, at least one of the following alternatives holds:
\begin{equation}\label{eq345}
 \begin{gathered}
 \text{there exists }j\text{ such that }y_j\in[0,\tau L],\\
 \text{or }y_j\in\bigcup_{d\in\mathcal D}B_d(L)\text{ for every }j.
 \end{gathered}
\end{equation}
Suppose that the first alternative holds, say $y_1\in[0,\tau L]$.
For every $d\in\mathcal D$, we have $0<d\le H$, and
\[
 \alpha H=\frac{\omega}{Z}H=\frac\omega p.
\]
Thus
\[
 B_d(L)
 \subseteq
 \left[0,\left(\frac\omega p+\rho\right)L\right]
 \qquad(d\in\mathcal D).
\]
Since
\[
 u=\frac{\omega-\eps/2}{h}
 <\frac\omega h<\frac\omega p,
\]
we obtain
\begin{equation}\label{eq346}
 n\le
 \left[
 u+\tau+(h-2)\left(\frac\omega p+\rho\right)
 \right]L.
\end{equation}
But
\begin{align*}
 &\omega-\frac\eps4
 -\left[
 u+\tau+(h-2)\left(\frac\omega p+\rho\right)
 \right]\\
 &=\frac{\omega}{hp}
 +\frac{\eps}{2h}-\frac\eps4
 -\tau-(h-2)\rho\\
 &\ge\frac{\omega}{hp}-h\rho-\frac\eps4\\
 &=\frac{\omega}{hp}-h\rho-\frac{\rho}{400}>0,
\end{align*}
where the last inequality follows from
\[
 \rho<\frac{\omega}{100h^2p}.
\]
Hence
\[
 n<\left(\omega-\frac\eps4\right)L\le\min K,
\]
a contradiction.

Therefore only the second alternative in \eqref{eq345} is possible, and hence
\[
 y_j\in\bigcup_{d\in\mathcal D}B_d(L)
 \qquad(1\le j\le h-s).
\]
By \eqref{eq341} and \eqref{eq324}, we may choose
\[
 \beta_i\in[hB_0,hB_0+m_q]
 \qquad(1\le i\le s),
\]
\[
 \gamma_j\in[pB_0-2,hB_0]
 \qquad(1\le j\le h-s),
\]
such that
\[
 x_i-a\equiv\beta_i\pmod q,
 \qquad
 y_j-a\equiv\gamma_j\pmod q.
\]
Put
\[
 D:=\sum_{i=1}^s\beta_i+
 \sum_{j=1}^{h-s}\gamma_j-phB_0.
\]
By \eqref{eq319},
\[
 0<D<q,
 \qquad
 n\equiv ha+\sum_{i=1}^s\beta_i+
 \sum_{j=1}^{h-s}\gamma_j
 \equiv t+D\not\equiv t\pmod q,
\]
again contradicting \eqref{eq342}.
Therefore
\[
 \left(K\cap(t+q\mathbb Z)\right)
 \cap
 \bigcup_{s=1}^h
 \left(s\mathcal E(L)+(h-s)\mathcal C\right)
 =\varnothing.
\]
\end{proof}

\begin{lemma}\label{lem35}
Assume that \eqref{eq310}, \eqref{eq21}, \eqref{eq22}, and \eqref{eq311} hold. Then there exist
\[
 J_1\subseteq J_0\subseteq K,\qquad
 V\subseteq B_C(L),\qquad U_i\subseteq B_{O_i}(L)\ (0\le i<p),
\]
such that
\begin{align}
 |J_0|&\asymp T,\qquad |J_1|\asymp T,\label{eq347}\\
 (J\setminus J_0)\cap(t+q\mathbb Z)
 &\subseteq V+U_0+\cdots+U_{p-1},\label{eq348}\\
 J_1\cap(V+U_0+\cdots+U_{p-1})&=\varnothing.\label{eq349}
\end{align}
Every $x\in V\cup\bigcup_iU_i$ satisfies
\begin{equation}\label{eq350}
\begin{aligned}
 x&\not\equiv e_\nu\pmod{q_\nu}\quad(1\le\nu\le k),\\
 x&\not\equiv e_{k+1}\pmod{q_{k+1}}
 \quad(k<N,\ q_{k+1}\le2L).
\end{aligned}
\end{equation}
If $p\ge2$, the sets may also be chosen so that
\begin{align}
 x&\equiv a+pB_0\pmod q\quad
 (x\in V\cup\bigcup_iU_i),\label{eq351}\\
 \sum_{i=0}^{p-1}|U_i|&=O(T^{1/p}),
 \qquad |V|=O(L/T).\label{eq352}
\end{align}
If $p=1$, they may also be chosen so that
\begin{align}
 x&\equiv a+B_0\pmod q\quad(x\in V\cup U_0),\label{eq353}\\
 |U_0|&=O(T/q),
 \qquad |V|=O(2^kL/T)=O(L^\theta).\label{eq354}
\end{align}
Moreover, one may choose
\begin{equation}\label{eq355}
 |J_0|\le\zeta T+1,
\end{equation}
and
\begin{equation}\label{eq356}
 |J_1|\ge
 \begin{cases}
 2^{-p-1}\zeta T,&p\ge2,\\
 \dfrac{\zeta}{2}T,&p=1.
 \end{cases}
\end{equation}
\end{lemma}

\begin{proof}
Choose an integer $f$ satisfying
\[
 f\equiv a+pB_0\pmod q.
\]
Since $h=p+1$ and $r\equiv a+(p+1)B_0\pmod q$,
\[
 \begin{aligned}
 (p+1)f
 &\equiv(p+1)(a+pB_0)\\
 &=(p+1)a+p(p+1)B_0\\
 &\equiv a+pr\equiv t\pmod q.
 \end{aligned}
\]
Also, since $q>pB_0$,
\[
 f-a\equiv pB_0\not\equiv0\pmod q,
 \qquad f\not\equiv a=e_j\pmod q.
\]

Suppose first that $p\ge2$. Apply Lemma~\ref{lem27} with
$d=p$ and $S_i=\{f\}$ $(0\le i<p)$. We obtain $R_0,\ldots,R_{p-1}$ satisfying
\[
 \begin{aligned}
 (R_0+\cdots+R_{p-1})\bmod Q
 &=\{u\bmod Q:u\equiv pf\pmod q\},\\
 |R_i|&\ll(Q/q)^{1/p},\\
 x&\equiv f\pmod q\quad(x\in R_i),\\
 R_i\cap(e_\nu+q_\nu\mathbb Z)&=\varnothing
 \quad(1\le\nu\le k,\ 0\le i<p).
 \end{aligned}
\]
Since $f\not\equiv e_j\pmod q$ and
$\mathbb Z_{q_\nu}\setminus\{e_\nu\}\ne\varnothing$ for $\nu\ne j$, the Chinese remainder theorem gives $z_0\in[0,Q-1]$ such that
\[
 z_0\equiv f\pmod q,\qquad
 z_0\not\equiv e_\nu\pmod{q_\nu}\quad(1\le\nu\le k).
\]
Put
\[
 d=p,\qquad c_*=z_0,\qquad R_V=\{0\},\qquad
 M=\left\lfloor(\zeta T/Q)^{1/p}\right\rfloor,
\]
\[
 S=QM^p,\qquad G=S-(p+1)Q,\qquad
 E_i=R_i+QM^i[0,M-1]\quad(0\le i<p).
\]
Since $Q\le\eta T$ and $\eta\le\zeta/2^p$,
\[
 (\zeta T/Q)^{1/p}\ge2,\qquad
 \tfrac12(\zeta T/Q)^{1/p}\le M\le(\zeta T/Q)^{1/p},
\]
and hence
\[
 2^{-p}\zeta T\le S\le\zeta T.
\]
By \eqref{eq215}--\eqref{eq217} in Lemma~\ref{lem25}, with $m=p$,
\[
 \begin{aligned}
 \relax[pQ,S-Q]\cap(pf+q\mathbb Z)
 &\subseteq E_0+\cdots+E_{p-1}
 \subseteq[0,S+(p-1)Q],\\
 |E_i|&=M|R_i|
 \ll(T/Q)^{1/p}(Q/q)^{1/p}
 =(T/q)^{1/p}.
 \end{aligned}
\]
Since $R_i\subseteq[0,Q-1]$ and $M^i(M-1)\le M^p-1$,
\[
 E_i\subseteq[0,Q-1]+Q[0,M^p-1]=[0,S-1].
\]
Taking
\[
 \eta\le\frac{2^{-p-2}\zeta}{p+1},
\]
we obtain
\[
 (p+1)Q\le\frac14S,\qquad
 \frac12S\le G=S-(p+1)Q\le S,\qquad G\asymp T.
\]
Since $c_*=z_0\equiv f\pmod q$,
\[
 t-c_*\equiv(p+1)f-f=pf\pmod q.
\]

Suppose now that $p=1$, so that $h=2$. Then
\[
 f\equiv a+B_0\pmod q,\qquad t\equiv2f\pmod q.
\]
Apply \eqref{eq226} in Lemma~\ref{lem27} with
$d=2$ and $S_0=S_1=\{f\}$, and denote the resulting sets by $R_V,R_U$. Then
\[
 (R_V+R_U)\bmod Q=\{u\bmod Q:u\equiv t\pmod q\},
 \qquad |R_V|=2^{k-1},\qquad |R_U|\le Q/q,
\]
and all elements of $R_V,R_U$ are congruent to $f$ modulo $q$ and satisfy \eqref{eq223}.
Take
\[
 d=2,\qquad c_*=0\in Q\mathbb Z,
\]
and put
\[
 M=\lfloor\zeta T/Q\rfloor,\qquad S=QM,\qquad
 G=S-3Q,\qquad E_0=R_U+Q[0,M-1].
\]
Since
\[
 \zeta T-Q\le S\le\zeta T,\qquad Q\le\eta T,\qquad
 \eta\le\zeta/100,
\]
we have
\[
 (\zeta-4\eta)T\le G=S-3Q\le\zeta T,\qquad G\asymp T.
\]
If $n\in[2Q,S-Q]\cap(t+q\mathbb Z)$, then by the definition of
$(R_V+R_U)\bmod Q$ there exist $r_V\in R_V$ and $r_U\in R_U$ such that
$n\equiv r_V+r_U\pmod Q$. Since $0\le r_V,r_U<Q$,
\[
 v=(n-r_V-r_U)/Q\in\mathbb Z,\qquad
 0\le v\le M-1,
\]
and
\[
 n=r_V+(r_U+Qv)\in R_V+E_0.
\]
Also,
\[
 E_0\subseteq[0,Q-1]+Q[0,M-1]=[0,S-1],
\]
\[
 R_V+E_0\subseteq[0,S+Q],\qquad |E_0|=M|R_U|.
\]
Since $c_*\in Q\mathbb Z$, we have $t-c_*\equiv t\pmod q$.

In both cases,
\[
 G=S-(d+1)Q\in Q\mathbb Z,\qquad
 G\asymp T,\qquad S\le\zeta T,
\]
and
\begin{equation}\label{eq357}
 [dQ,S-Q]\cap(t-c_*+q\mathbb Z)
 \subseteq R_V+\sum_{i=0}^{p-1}E_i
 \subseteq[0,S+(d-1)Q].
\end{equation}
For each $0\le i<p$, choose
\[
 H_i\in Q\mathbb Z,\qquad |H_i-\alpha O_iL|\le Q,
\]
and put
\[
 H_*'=\sum_{i=0}^{p-1}H_i,\qquad
 \widetilde U_i=H_i+E_i,\qquad
 c_m=c_*+mG,\qquad \widetilde V_m=c_m+R_V,
\]
\[
 b_*=c_*+H_*'+dQ,\qquad
 I_m=[c_m+H_*'+dQ,c_m+H_*'+S-Q]
     =[b_*+mG,b_*+(m+1)G].
\]
If $n\in I_m\cap(t+q\mathbb Z)$, put
\[
 u=n-c_m-H_*'.
\]
Since $G,H_*'\in Q\mathbb Z$,
\[
 dQ\le u\le S-Q,\qquad u\equiv t-c_*\pmod q.
\]
Thus, by \eqref{eq357},
\[
 u\in R_V+\sum_iE_i,\qquad
 n\in\widetilde V_m+\sum_i\widetilde U_i.
\]
Let
\[
 \mathcal F_m=\widetilde V_m+\sum_{i=0}^{p-1}\widetilde U_i,
 \qquad \mathcal M=\{m\in\mathbb Z:I_m\cap J\ne\varnothing\}.
\]
By the upper bound in \eqref{eq357} and $G=S-(d+1)Q$,
\[
 I_m\cap(t+q\mathbb Z)\subseteq\mathcal F_m
 \subseteq[c_m+H_*',c_m+H_*'+S+(d-1)Q]
 =[b_*+mG-dQ,b_*+(m+1)G+dQ].
\]

Since $S\le\zeta T$, $Q\le\eta T$, and $T\le\sigma L$, the already fixed bounds in \eqref{eq39} give
\[
 \eps+\zeta\sigma+(2h+3)\eta\sigma
 \le\eps+\sigma/2+\rho/8<\rho/2,
\]
and hence
\[
 \eps L+S+(2h+3)Q\le\rho L/2.
\]
Since $E_i\subseteq[0,S-1]$,
\[
 |x-\alpha O_iL|
 \le S+Q\le\rho L-(h+3)Q
 \qquad(x\in\widetilde U_i).
\]
Since
\[
 C+\sum_{i=0}^{p-1}O_i=Z,\qquad \alpha Z=\omega,
\]
we have
\[
 \alpha CL+\sum_{i=0}^{p-1}\alpha O_iL=\omega L.
\]
If $m\in\mathcal M$, choose $n\in I_m\cap J$ and write
\[
 n=c_m+H_*'+u,\qquad dQ\le u\le S-Q.
\]
For $x\in\widetilde V_m$, since $R_V\subseteq[0,Q-1]$,
\[
 \begin{aligned}
 |x-\alpha CL|
 &\le |n-\omega L|
 +\left|H_*'-\sum_{i=0}^{p-1}\alpha O_iL\right|
 +|u|+|x-c_m|\\
 &\le\eps L+pQ+S+Q
 =\eps L+hQ+S.
 \end{aligned}
\]
Therefore
\[
 |x-\alpha CL|
 \le\eps L+hQ+S\le\rho L-(h+3)Q
 \qquad(x\in\widetilde V_m).
\]
For
\[
 b_D^-:=\left\lceil\alpha DL-\rho L\right\rceil,
 \qquad
 b_D^+:=\left\lfloor\alpha DL+\rho L\right\rfloor
 \qquad(D\in\{C,O_0,\ldots,O_{p-1}\}),
\]
we have
\[
 B_D(L)\cap\mathbb Z=[b_D^-,b_D^+]\cap\mathbb Z.
\]
By the preceding estimates and $Q\ge1$,
\[
 \begin{aligned}
 x-b_{O_i}^-,\ b_{O_i}^+-x
 &\ge(h+3)Q-1\ge(h+2)Q
 &&(x\in\widetilde U_i,\ 0\le i<p),\\
 x-b_C^-,\ b_C^+-x
 &\ge(h+3)Q-1\ge(h+2)Q
 &&(x\in\widetilde V_m,\ m\in\mathcal M).
 \end{aligned}
\]
Hence
\[
 \widetilde U_i\subseteq
 [b_{O_i}^-+(h+2)Q,b_{O_i}^+-(h+2)Q],
 \qquad 0\le i<p,
\]
and
\[
 \widetilde V_m\subseteq
 [b_C^-+(h+2)Q,b_C^+-(h+2)Q],
 \qquad m\in\mathcal M.
\]

Since
\[
 |K|=\frac\eps2L+O(1),\qquad
 G\le\zeta T\le\zeta\sigma L\le\frac\eps{24}L,
\]
for all sufficiently large $L$,
\[
 6G\le\frac\eps4L<|K|,
 \qquad K\subseteq J.
\]
Apply Lemma~\ref{lem29} with
\[
 \mathcal T=t+q\mathbb Z,
 \qquad b_*=c_*+H_*'+dQ,
 \qquad G=S-(d+1)Q,
 \qquad
 \mathcal F_m=\widetilde V_m+\sum_{i=0}^{p-1}\widetilde U_i.
\]
Let $m_0\in\mathcal M$ be the resulting integer, and put
\[
 A_*=b_*+m_0G,\qquad J_0=I_{m_0}=[A_*,A_*+G]\subseteq K,
\]
\[
 \widetilde V=\bigcup_{m\in\mathcal M\setminus\{m_0\}}\widetilde V_m,
 \qquad
 \mathcal F=\bigcup_{m\in\mathcal M\setminus\{m_0\}}\mathcal F_m
 =\widetilde V+\sum_{i=0}^{p-1}\widetilde U_i.
\]
By \eqref{eq231} and $G\asymp T$,
\[
 |\mathcal M|\le2+|J|/G=O(L/T).
\]

Since $G,H_i\in Q\mathbb Z$,
\[
 \widetilde U_i\bmod Q=
 \begin{cases}
 R_i,&p\ge2,\\
 R_U,&p=1,\ i=0,
 \end{cases}
 \qquad
 \widetilde V_m\bmod Q=
 \begin{cases}
 \{z_0\},&p\ge2,\\
 R_V,&p=1.
 \end{cases}
\]
Therefore every $x\in\widetilde V\cup\bigcup_i\widetilde U_i$ satisfies
\[
 x\equiv f\pmod q,
 \qquad
 x\not\equiv e_\nu\pmod{q_\nu}\quad(1\le\nu\le k).
\]
Apply Lemma~\ref{lem28}\textup{(ii)}, with $m=h$, to
$\widetilde V$ and the sets $\widetilde U_i$ $(0\le i<p)$, and denote the resulting sets by $V$ and $U_i$.
Then
\[
 V\subseteq B_C(L),\qquad U_i\subseteq B_{O_i}(L),
\]
and every $x\in V\cup\bigcup_iU_i$ satisfies \eqref{eq350} and
\[
 x\equiv f\equiv a+pB_0\pmod q.
\]
Thus \eqref{eq351} holds when $p\ge2$, and \eqref{eq353} holds when $p=1$.
Put
\[
 \mathcal H=V+\sum_{i=0}^{p-1}U_i,\qquad\Delta=h(h+1)Q.
\]
By \eqref{eq229},
\[
 \mathcal F
 =\widetilde V+\sum_{i=0}^{p-1}\widetilde U_i
 \subseteq V+\sum_{i=0}^{p-1}U_i
 =\mathcal H.
\]
Moreover, for any $y\in\mathcal H$, there exists $y_0\in\mathcal F$ such that
\[
 |y-y_0|\le h(h+1)Q=\Delta.
\]
Hence
\[
 \mathcal F\subseteq\mathcal H
 \subseteq\mathcal F+[-\Delta,\Delta].
\]
Lemma~\ref{lem29}, with
\[
 J_1=[A_*+(d+1)Q+\Delta,A_*+G-(d+1)Q-\Delta],
\]
then gives
\[
 J_1\subseteq J_0\subseteq K,\qquad
 (J\setminus J_0)\cap(t+q\mathbb Z)\subseteq\mathcal H,
 \qquad J_1\cap\mathcal H=\varnothing.
\]

Put
\[
 D_h=3(d+1)+2h(h+1).
\]
Since $G=S-(d+1)Q$,
\[
 G-2(d+1)Q-2\Delta=S-D_hQ.
\]
If $p\ge2$, then $d=p$ and
\[
 D_h=2h^2+5h<(2h+3)(h+1).
\]
Thus \eqref{eq39} already gives
\[
 \eta\le\frac{2^{-p-1}\zeta}{D_h},
\]
and therefore
\[
 S-D_hQ\ge(2^{-p}\zeta-D_h\eta)T
 \ge2^{-p-1}\zeta T>0.
\]
If $p=1$, then $d=2$ and $h=2$, so
\[
 \Delta=6Q,\qquad (d+1)Q+\Delta=9Q,\qquad
 D_h=3\cdot3+2\cdot2\cdot3=21,
\]
and
\[
 J_1=[A_*+9Q,A_*+G-9Q].
\]
Since $\eta\le\zeta/100$,
\[
 \begin{aligned}
 S-D_hQ
 &=S-21Q\\
 &\ge\zeta T-22Q\\
 &\ge(\zeta-22\eta)T\\
 &\ge\frac\zeta2T>0.
 \end{aligned}
\]
Thus in both cases $G>2(d+1)Q+2\Delta$. The two estimates above, together with
\[
 |J_1|=S-D_hQ+1,
\]
give \eqref{eq356}. In particular,
\[
 S-D_hQ\ge2^{-p-1}\zeta T.
\]
Hence, by \eqref{eq234},
\[
 |J_0|=G+1\asymp T,\qquad
 |J_1|=S-D_hQ+1\gg T.
\]
Since
\[
 |J_1|\le|J_0|=G+1\le\zeta T+1,
\]
we also have $|J_1|\asymp T$.

If $p\ge2$, then $|R_V|=1$. By \eqref{eq229} and $|E_i|\ll(T/q)^{1/p}$,
\[
 \sum_{i=0}^{p-1}|U_i|
 \ll\sum_{i=0}^{p-1}|\widetilde U_i|
 =\sum_{i=0}^{p-1}|E_i|=O(T^{1/p}),
\]
and
\[
 |V|\ll|\widetilde V|
 \le|\mathcal M||R_V|=O(L/T).
\]
If $p=1$, then $|R_V|=2^{k-1}$. Similarly,
\[
 |U_0|\ll|E_0|=M|R_U|
 \le(\zeta T/Q)(Q/q)=O(T/q),
\]
and
\[
 |V|\ll|\mathcal M||R_V|=O(2^kL/T).
\]
In this case Lemma~\ref{lem21} gives $k=O(\sqrt{\log L})$, and hence
\[
 2^k=L^{o(1)}.
\]
Since $T\asymp L^\theta$ and $\theta>1/2$,
\[
 2^kL/T=L^{1-\theta+o(1)}=O(L^\theta),
 \qquad |V|=O(L^\theta).
\]
\end{proof}

\begin{lemma}\label{lem36}
Assume that \eqref{eq310}, \eqref{eq21}, \eqref{eq22}, and
\eqref{eq311} hold, and that $p=1$ (hence $h=2$).
Let $\mathcal E(L)$ be the set given by Lemma~\ref{lem34}.
Suppose further that
\[
 D_{\rm o}\subseteq[\delta L,\rho L],\qquad
 H_0^{\rm o}\subseteq B_H(L),
\]
and that every $x\in D_{\rm o}\cup H_0^{\rm o}$ satisfies
\[
 \begin{aligned}
 x&\not\equiv e_\nu\pmod{q_\nu}\quad(1\le\nu\le k),\\
 x&\not\equiv e_{k+1}\pmod{q_{k+1}}
 \quad(k<N,\ q_{k+1}\le2L),
 \end{aligned}
\]
as well as
\[
 x\equiv r\pmod q\qquad(x\in H_0^{\rm o}).
\]
Let
\[
 J_1\subseteq J_0\subseteq K,\qquad
 V\subseteq B_C(L),\qquad U_0\subseteq B_{O_0}(L)
\]
be the sets supplied by Lemma~\ref{lem35} for $p=1$; thus they satisfy
\eqref{eq347}--\eqref{eq350} and
\eqref{eq353}--\eqref{eq354}.
Put
\[
 I=J_0,\qquad I'=J_1.
\]
Then there exist
\[
 D_{\rm r}\subseteq B_H(L),\qquad
 X\subseteq B_C(L),\qquad
 Y\subseteq B_{O_0}(L)
\]
such that
\begin{equation}\label{eq358}
 |D_{\rm r}|+|X|+|Y|=O(T).
\end{equation}
Every $z\in D_{\rm r}\cup X\cup Y$ satisfies
\begin{equation}\label{eq359}
\begin{aligned}
 z&\not\equiv e_i\pmod{q_i}\quad(1\le i\le k),\\
 z&\not\equiv e_{k+1}\pmod{q_{k+1}}
 \quad(k<N,\ q_{k+1}\le2L).
\end{aligned}
\end{equation}
Moreover,
\begin{equation}\label{eq360}
 I\cap(t+q\mathbb Z)\subseteq(a+D_{\rm r})\cup(X+Y).
\end{equation}
Let
\[
 \mathcal A^{(1)}:=A_0\cup\mathcal E(L)\cup D_{\rm o}\cup H_0^{\rm o}
 \cup V\cup U_0\cup D_{\rm r}\cup X\cup Y.
\]
Then there exists $\mathcal P_a(L)\subseteq I'$ such that
\begin{align}
 |\mathcal P_a(L)|&\gg T/q,\label{eq361}\\
 \mathcal P_a(L)&\subseteq
 2\mathcal A^{(1)}\setminus2(\mathcal A^{(1)}\setminus\{a\}).\label{eq362}
\end{align}
\end{lemma}

\begin{proof}
Since $p=1$ and $h=2$, we retain the notation introduced above:
\[
 t\equiv2a+2B_0\pmod q,
 \qquad
 r\equiv a+2B_0\pmod q,
 \qquad
 q=q_j,
 \qquad
 e_j=a.
\]
By \eqref{eq347} and \eqref{eq349},
\[
 I'=J_1\subseteq I=J_0\subseteq K,
 \qquad
 |I|\asymp T,
 \qquad
 |I'|\asymp T,
\]
and $V\subseteq B_C(L)$, $U_0\subseteq B_{O_0}(L)$ satisfy
\[
 x\equiv a+B_0\pmod q\quad(x\in V\cup U_0),
 \qquad I'\cap(V+U_0)=\varnothing.
\]

For $i\le k$, $i\ne j$, put
\[
 c_i\equiv a-e_i\pmod{q_i}.
\]
Since $q_i>q_0>6$, choose successively
\[
 \rho_i\in\mathbb Z_{q_i}\setminus\{0,c_i\},
 \qquad
 \beta_i\in\mathbb Z_{q_i}\setminus\{0,c_i,\rho_i\}.
\]
Put
\[
 H_i=\{c_i,\rho_i\},
 \qquad
 B_i=H_i-\beta_i
 =\{c_i-\beta_i,\rho_i-\beta_i\}.
\]
Since $\beta_i\notin\{c_i,\rho_i\}$,
\[
 0\notin B_i,
 \qquad
 |H_i|=|B_i|=2.
\]
Define
\[
 \widehat{\mathcal G}=
 \left\{n\in t+q\mathbb Z:
 n-2e_i\notin H_i\pmod{q_i}
 \quad(i\le k,\ i\ne j)\right\},
 \qquad \mathcal G=I\cap\widehat{\mathcal G}.
\]
If
\[
 n\in(I\cap(t+q\mathbb Z))\setminus\mathcal G,
\]
then by the definition of $\widehat{\mathcal G}$ there exists $i\le k$, $i\ne j$, such that
\[
 n-2e_i\in H_i\pmod{q_i}.
\]

For every $n\in(I\cap(t+q\mathbb Z))\setminus\mathcal G$, fix such an index $i=i(n)$ and choose
\[
 h_i\in H_i,\qquad n-2e_i\equiv h_i\pmod{q_i}.
\]
Modulo $q_i$, impose
\[
 x-e_i\equiv\beta_i,\qquad
 y-e_i\equiv h_i-\beta_i.
\]
By the choice of $\beta_i$,
\[
 \beta_i\ne0,\qquad
 h_i-\beta_i\in B_i\subseteq\mathbb Z_{q_i}\setminus\{0\}.
\]
Thus $x,y\not\equiv e_i\pmod{q_i}$, and
\[
 \begin{aligned}
 x+y-2e_i
 &\equiv\beta_i+(h_i-\beta_i)
 \equiv h_i
 \equiv n-2e_i\pmod{q_i}.
 \end{aligned}
\]

For $\nu\le k$, $\nu\ne i(n),j$, put
\[
 s_\nu\equiv n-2e_\nu\pmod{q_\nu}.
\]
Since
\[
 |B_\nu|=2,\qquad
 |\{b\in B_\nu:s_\nu-b=0\}|\le1,
\]
we may choose $b_\nu\in B_\nu\setminus\{s_\nu\}$ and impose
\[
 y-e_\nu\equiv b_\nu,\qquad
 x-e_\nu\equiv s_\nu-b_\nu\pmod{q_\nu},
 \qquad s_\nu-b_\nu\ne0.
\]
Since $0\notin B_\nu$,
\[
 x\not\equiv e_\nu\pmod{q_\nu},
 \qquad
 y\not\equiv e_\nu\pmod{q_\nu}.
\]
Also,
\[
 \begin{aligned}
 x+y-2e_\nu
 &\equiv(s_\nu-b_\nu)+b_\nu\\
 &\equiv s_\nu\\
 &\equiv n-2e_\nu
 \pmod{q_\nu}.
 \end{aligned}
\]
Hence, for $1\le\nu\le k$, $\nu\ne j$,
\[
 x+y\equiv n\pmod{q_\nu},
 \qquad
 x,y\not\equiv e_\nu\pmod{q_\nu}.
\]

Modulo $q=q_j$, impose
\[
 x\equiv a+B_0+1\pmod q,
 \qquad
 y\equiv a+B_0-1\pmod q.
\]
Since $q>B_0+2$,
\[
 a+B_0\pm1\not\equiv a\pmod q.
\]
Moreover,
\[
 \begin{aligned}
 x+y
 &\equiv(a+B_0+1)+(a+B_0-1)\\
 &\equiv2a+2B_0\\
 &\equiv t\\
 &\equiv n
 \pmod q.
 \end{aligned}
\]
By the Chinese remainder theorem, there exists a pair of residue classes modulo $Q$,
\[
 (\xi_n,\upsilon_n)\pmod Q,
\]
such that
\[
 \xi_n+\upsilon_n\equiv n\pmod Q,
\]
and
\[
 \xi_n,\upsilon_n\not\equiv e_\nu\pmod{q_\nu}
 \qquad(1\le\nu\le k).
\]

Define
\[
 \mathcal R_c^{(X)}
 =\{\xi_n:n\in(I\cap(t+q\mathbb Z))\setminus\mathcal G\},
\]
\[
 \mathcal R_c^{(Y)}
 =\{\upsilon_n:n\in(I\cap(t+q\mathbb Z))\setminus\mathcal G\}.
\]
By the choice of $i(n)$,
\[
 \xi_n-e_{i(n)}\equiv\beta_{i(n)}\pmod{q_{i(n)}}.
\]
By the choices made for the $y$-coordinate, for every
$\upsilon_m\in\mathcal R_c^{(Y)}$ and every $i\le k$, $i\ne j$,
\[
 \upsilon_m-e_i\in B_i\pmod{q_i}.
\]
For arbitrary $n,m$, let $i=i(n)$. Then
\[
 \xi_n+\upsilon_m-2e_i
 \in\beta_i+B_i
 =H_i
 \pmod{q_i}.
\]
Hence
\[
 \xi_n+\upsilon_m\notin\widehat{\mathcal G}.
\]

Put
\[
 K_d=[\alpha dL-2\rho L/3,\alpha dL+2\rho L/3]
 \qquad(d=C,O_0,H).
\]
By definition,
\[
 K_d\subseteq B_d(L).
\]
Since $p=1$, we have $\alpha H=\omega$ and $\alpha C+\alpha O_0=\omega$.
For $n\in I\subseteq K$, using $a\le\rho T/10$ and $T\le\sigma L$,
\[
 |(n-a)-\alpha HL|
 \le|n-\omega L|+a
 \le\left(\frac\eps4+\frac{\rho\sigma}{10}\right)L
 <\frac{2\rho L}{3}.
\]
Thus $I-a\subseteq K_H$.
For $n\in I\subseteq K$, since $\alpha C+\alpha O_0=\omega$,
\[
 |K_C\cap(n-K_{O_0})|
 \ge\left(\frac{4\rho}{3}-\frac\eps4\right)L-2.
\]
Since $\eps=\rho/100$, $Q\le\eta T\le\eta\sigma L$, and
\[
 \eta\sigma\le\frac{\rho}{8(2h+3)}=\frac{\rho}{56}
 \qquad(h=2),
\]
we have
\[
 \begin{aligned}
 |K_C\cap(n-K_{O_0})|-(3Q+1)
 &\ge\left(\frac{4\rho}{3}-\frac{\rho}{400}-3\eta\sigma\right)L-3\\
 &\ge\left(\frac43-\frac1{400}-\frac3{56}\right)\rho L-3>0
 \end{aligned}
\]
for all sufficiently large $L$. Thus
\[
 |K_C\cap(n-K_{O_0})|\ge3Q+1.
\]
For $n\in(I\cap(t+q\mathbb Z))\setminus\mathcal G$, we have
\[
 |K_C\cap(n-K_{O_0})|\ge3Q+1,
 \qquad \xi_n+\upsilon_n\equiv n\pmod Q,
\]
and
\[
 \xi_n,\upsilon_n\not\equiv e_\nu\pmod{q_\nu}
 \qquad(1\le\nu\le k).
\]
Thus Lemma~\ref{lem28}\textup{(iii)}, with
$K_0=K_C$ and $K_1=K_{O_0}$, gives
\[
 x_n+y_n=n,\qquad x_n\in K_C,\quad y_n\in K_{O_0},
 \qquad x_n\equiv\xi_n,\quad y_n\equiv\upsilon_n\pmod Q.
\]
If $k<N$ and $q_{k+1}\le2L$, it also gives
\[
 x_n,y_n\not\equiv e_{k+1}\pmod{q_{k+1}}.
\]
Define
\[
 X_c=\{x_n:n\in(I\cap(t+q\mathbb Z))\setminus\mathcal G\},
\]
\[
 Y_c=\{y_n:n\in(I\cap(t+q\mathbb Z))\setminus\mathcal G\}.
\]
Since $x_n+y_n=n$,
\[
 (I\cap(t+q\mathbb Z))\setminus\mathcal G
 \subseteq X_c+Y_c.
\]
Since $x_n\equiv\xi_n$ and $y_m\equiv\upsilon_m\pmod Q$, for arbitrary
$x_n\in X_c$ and $y_m\in Y_c$, with $i=i(n)$,
\[
 x_n+y_m-2e_i
 \equiv\xi_n+\upsilon_m-2e_i
 \in H_i\pmod{q_i}.
\]
Also $x_n+y_m\equiv2a+2B_0\equiv t\pmod q$. Hence
\[
 X_c+Y_c
 \subseteq(t+q\mathbb Z)\setminus\widehat{\mathcal G},
 \qquad (X_c+Y_c)\cap\mathcal G=\varnothing.
\]

For $n\in\mathcal G$, put
\[
 d=n-a.
\]
Since $n\equiv t\pmod q$,
\[
 d\equiv t-a\equiv a+2B_0\equiv r\pmod q.
\]
Since $q>2B_0$,
\[
 r-a\equiv2B_0\not\equiv0\pmod q,
 \qquad d\not\equiv a=e_j\pmod q.
\]
If, for some $i\le k$, $i\ne j$,
\[
 d\equiv e_i\pmod{q_i},
\]
then, since $d=n-a$,
\[
 \begin{aligned}
 n-2e_i
 &\equiv a+d-2e_i\\
 &\equiv a-e_i\\
 &\equiv c_i
 \pmod{q_i}.
 \end{aligned}
\]
This contradicts $c_i\in H_i$ and $n\in\mathcal G$. Therefore
\[
 d\not\equiv e_i\pmod{q_i}
 \qquad(1\le i\le k).
\]

Define
\[
 \mathcal G_0=
 \begin{cases}
 \{n\in\mathcal G:n-a\not\equiv e_{k+1}\pmod{q_{k+1}}\},
 &k<N,\ q_{k+1}\le2L,\\
 \mathcal G,&\text{otherwise}.
 \end{cases}
\]
and put
\[
 D_{\rm r}=\{n-a:n\in\mathcal G_0\}.
\]
Since $I-a\subseteq K_H\subseteq B_H(L)$,
\[
 D_{\rm r}\subseteq B_H(L).
\]
By the definition of $\mathcal G_0$, every $d\in D_{\rm r}$ satisfies
\[
 \begin{cases}
 d\not\equiv e_i\pmod{q_i},&1\le i\le k,\\
 d\not\equiv e_{k+1}\pmod{q_{k+1}},&k<N,\ q_{k+1}\le2L.
 \end{cases}
\]
Moreover,
\[
 \mathcal G_0\subseteq a+D_{\rm r}.
\]

Let $n\in\mathcal G\setminus\mathcal G_0$. Then, by the definition of $\mathcal G_0$,
\[
 k<N,
 \qquad
 \ell=q_{k+1}\le2L,
 \qquad
 n-a\equiv e_{k+1}\pmod\ell.
\]
Modulo $q$, impose
\[
 x\equiv a+B_0+2\pmod q,
 \qquad
 y\equiv a+B_0-2\pmod q.
\]
Then
\[
 x+y\equiv2a+2B_0\equiv t\equiv n\pmod q,
 \qquad x,y\not\equiv a=e_j\pmod q.
\]
For $i\le k$, $i\ne j$, in order to have
\[
 x\not\equiv e_i\pmod{q_i},
 \qquad
 n-x\not\equiv e_i\pmod{q_i},
\]
observe that
\[
 |\{e_i,n-e_i\}\bmod q_i|\le2<q_i.
\]
Hence we may choose
\[
 u_i\in\mathbb Z_{q_i}\setminus\{e_i,n-e_i\},
\]
and impose
\[
 x\equiv u_i\pmod{q_i},
 \qquad
 y\equiv n-u_i\pmod{q_i}.
\]
By the Chinese remainder theorem, there exist residue classes modulo $Q$,
\[
 (\xi_n',\upsilon_n')\pmod Q,
\]
such that
\[
 \xi_n'+\upsilon_n'\equiv n\pmod Q,
\]
and
\[
 \xi_n',\upsilon_n'\not\equiv e_i\pmod{q_i}
 \qquad(1\le i\le k).
\]

Since $\xi_n'+\upsilon_n'\equiv n\pmod Q$ and
$|K_C\cap(n-K_{O_0})|\ge3Q+1$, Lemma~\ref{lem28}\textup{(iii)}, with
$K_0=K_C$ and $K_1=K_{O_0}$, gives
\[
 x_n'+y_n'=n,\qquad x_n'\in K_C,\quad y_n'\in K_{O_0},
 \qquad x_n'\equiv\xi_n',\quad y_n'\equiv\upsilon_n'\pmod Q,
\]
and both $x_n'$ and $y_n'$ satisfy \eqref{eq24}.
Define
\[
 X_t=\{x_n':n\in\mathcal G\setminus\mathcal G_0\},
 \qquad
 Y_t=\{y_n':n\in\mathcal G\setminus\mathcal G_0\}.
\]
Since $x_n'+y_n'=n$,
\[
 \mathcal G\setminus\mathcal G_0\subseteq X_t+Y_t.
\]

Put
\[
 X=X_c\cup X_t,
 \qquad
 Y=Y_c\cup Y_t.
\]
Since
\[
 X_c\cup X_t\subseteq K_C\subseteq B_C(L),
 \qquad
 Y_c\cup Y_t\subseteq K_{O_0}\subseteq B_{O_0}(L),
\]
we have
\[
 X\subseteq B_C(L),
 \qquad
 Y\subseteq B_{O_0}(L).
\]
Every
\[
 z\in D_{\rm r}\cup X\cup Y
\]
satisfies
\[
 \begin{cases}
 z\not\equiv e_i\pmod{q_i},&1\le i\le k,\\
 z\not\equiv e_{k+1}\pmod{q_{k+1}},&k<N,\ q_{k+1}\le2L.
 \end{cases}
\]
Also,
\[
 I\cap(t+q\mathbb Z)
 =\bigl((I\cap(t+q\mathbb Z))\setminus\mathcal G\bigr)
 \cup\mathcal G_0
 \cup(\mathcal G\setminus\mathcal G_0),
\]
and
\[
 \begin{gathered}
 (I\cap(t+q\mathbb Z))\setminus\mathcal G\subseteq X_c+Y_c,\\
 \mathcal G_0\subseteq a+D_{\rm r},\qquad
 \mathcal G\setminus\mathcal G_0\subseteq X_t+Y_t.
 \end{gathered}
\]
Therefore
\[
 \begin{aligned}
 I\cap(t+q\mathbb Z)
 &\subseteq(X_c+Y_c)\cup(a+D_{\rm r})\cup(X_t+Y_t)\\
 &\subseteq(a+D_{\rm r})\cup(X+Y).
 \end{aligned}
\]
Since
\[
 \begin{gathered}
 |D_{\rm r}|\le|\mathcal G_0|,\\
 |X_c|,|Y_c|\le|(I\cap(t+q\mathbb Z))\setminus\mathcal G|,\\
 |X_t|,|Y_t|\le|\mathcal G\setminus\mathcal G_0|,
 \end{gathered}
\]
we get
\[
 \begin{aligned}
 |D_{\rm r}|+|X|+|Y|
 &\le |\mathcal G_0|
 +2|(I\cap(t+q\mathbb Z))\setminus\mathcal G|
 +2|\mathcal G\setminus\mathcal G_0|\\
 &\le2|I\cap(t+q\mathbb Z)|\\
 &\le2|I|\\
 &=O(T).
 \end{aligned}
\]

Put
\[
 \mathcal P_0=I'\cap\mathcal G_0,
 \qquad
 \mathcal P_a(L)=\mathcal P_0\setminus(X_t+Y_t).
\]
If $n\in\mathcal P_a(L)\subseteq\mathcal G_0$, then
\[
 n-a\in D_{\rm r},
\]
and hence
\[
 n=a+(n-a)\in2\mathcal A^{(1)}.
\]
It remains to exclude
\[
 n\in2(\mathcal A^{(1)}\setminus\{a\}).
\]

Let $\mathcal C=\mathcal A^{(1)}\setminus\mathcal E(L)$. Since
\[
\begin{gathered}
 A_0\subseteq[0,L/R]\subseteq[0,\tau L],\qquad
 D_{\rm o}\subseteq[\delta L,\rho L]\subseteq[0,\tau L],\\
 H_0^{\rm o}\cup D_{\rm r}\subseteq B_H(L),\qquad
 V\cup X\subseteq B_C(L),\qquad
 U_0\cup Y\subseteq B_{O_0}(L),
\end{gathered}
\]
we have
\[
\begin{aligned}
 \mathcal C
 &\subseteq A_0\cup D_{\rm o}\cup H_0^{\rm o}\cup V\cup U_0
 \cup D_{\rm r}\cup X\cup Y\\
 &\subseteq[0,\tau L]\cup B_H(L)\cup B_C(L)\cup B_{O_0}(L).
\end{aligned}
\]
By the hypothesis on $H_0^{\rm o}$, \eqref{eq353}, and the definitions modulo $q$ of
$D_{\rm r},X,Y$,
\[
 \begin{aligned}
 x-a&\equiv2B_0\pmod q
 &&(x\in\mathcal C\cap B_H(L)),\\
 x-a&\in\{B_0,B_0+1,B_0+2\}\pmod q
 &&(x\in\mathcal C\cap B_C(L)),\\
 x-a&\in\{B_0,B_0-1,B_0-2\}\pmod q
 &&(x\in\mathcal C\cap B_{O_0}(L)).
 \end{aligned}
\]
Thus
\[
 (x-a)\bmod q\in
 \bigl(\{B_0+u:-2\le u\le2\}\cup\{2B_0\}\bigr)\bmod q
 \subseteq[B_0-2,2B_0]\bmod q
 =\mathcal R_q.
\]
Hence \eqref{eq323} and \eqref{eq324} hold.
By \eqref{eq325} in Lemma~\ref{lem34}, if
\[
 n\in K\cap(t+q\mathbb Z),
\]
then no sum of two elements, one of which lies in $\mathcal E(L)$, can equal $n$.
Since
\[
 \mathcal P_a(L)\subseteq I'\subseteq K,
\]
it remains only to consider $2\mathcal C$.

Since
\[
 K\subseteq J,\qquad
 \mathcal A^{(1)}\setminus\mathcal E(L)
 \subseteq[0,\tau L]\cup B_H(L)\cup B_C(L)\cup B_{O_0}(L),
\]
by \eqref{eq36} there are only two possible interval combinations inside $K$:
\[
 [0,\tau L]+B_H(L),\qquad B_C(L)+B_{O_0}(L).
\]
From the preceding congruence relations,
\[
 y\equiv a+2B_0\equiv r\pmod q
 \quad(y\in\mathcal C\cap B_H(L)).
\]
If
\[
 n=x+y,
 \qquad
 x\in\mathcal C\cap[0,\tau L],
 \qquad
 y\in\mathcal C\cap B_H(L),
\]
then $n\equiv t\pmod q$ gives
\[
 \begin{aligned}
 x
 &\equiv t-r\\
 &\equiv(2a+2B_0)-(a+2B_0)\\
 &\equiv a
 \pmod q.
 \end{aligned}
\]
Since $a=e_j$, $q=q_j$, $j\le k$, and by
\eqref{eq22}, \eqref{eq321}, the residue-avoidance hypotheses on
$D_{\rm o}\cup H_0^{\rm o}$, \eqref{eq350}, and \eqref{eq359},
\[
 A_0\cap(a+q\mathbb Z)=\{a\},
 \qquad
 (\mathcal A^{(1)}\setminus A_0)\cap(a+q\mathbb Z)=\varnothing.
\]
Therefore
\[
 \mathcal A^{(1)}\cap(a+q\mathbb Z)=\{a\},
 \qquad x=a.
\]

For the other interval combination, \eqref{eq353} and \eqref{eq349} give
\[
 \begin{aligned}
 x\bmod q&\in\{a+B_0,a+B_0+1,a+B_0+2\}
 \quad(x\in\mathcal C\cap B_C(L)),\\
 y\bmod q&\in\{a+B_0,a+B_0-1,a+B_0-2\}
 \quad(y\in\mathcal C\cap B_{O_0}(L)).
 \end{aligned}
\]
Let
\[
 x\equiv a+B_0+u\pmod q,
 \qquad
 y\equiv a+B_0-v\pmod q,
\]
where
\[
 u,v\in\{0,1,2\}.
\]
Then
\[
 x+y\equiv t+(u-v)\pmod q.
\]
If $u\ne v$, then, since $u,v\in\{0,1,2\}$ and $q>2$,
\[
 0<|u-v|\le2<q,
 \qquad t+(u-v)\not\equiv t\pmod q.
\]
Hence $x+y=n\equiv t\pmod q$ implies
\[
 u=v.
\]

If $u=v=0$, then by \eqref{eq349},
\[
 x+y\in V+U_0,\qquad
 n\in I',\qquad I'\cap(V+U_0)=\varnothing,
\]
so $x+y\ne n$.

If $u=v=1$, then
\[
 x+y\in X_c+Y_c,\qquad
 n\in\mathcal P_0\subseteq\mathcal G,\qquad
 (X_c+Y_c)\cap\mathcal G=\varnothing,
\]
so again $x+y\ne n$.

If $u=v=2$, then by the definition of $\mathcal P_a(L)$,
\[
 x+y\in X_t+Y_t,\qquad n\notin X_t+Y_t,
\]
and hence $x+y\ne n$.

Therefore
\[
 n\notin2(\mathcal A^{(1)}\setminus\{a\}),\qquad n\in2\mathcal A^{(1)},
\]
which gives
\[
 \mathcal P_a(L)
 \subseteq2\mathcal A^{(1)}\setminus2(\mathcal A^{(1)}\setminus\{a\}).
\]

Modulo $q=q_j$, the condition
\[
 n\equiv t\pmod q
\]
fixes one residue class. For $i\le k$, $i\ne j$, since $|H_i|=2$, the condition
\[
 n-2e_i\notin H_i\pmod{q_i}
\]
leaves $q_i-2$ residue classes. Put
\[
 S_j=\{t\}\subseteq\mathbb Z_q,
 \qquad
 S_i=\mathbb Z_{q_i}\setminus(2e_i+H_i)
 \quad(i\le k,\ i\ne j).
\]
Then
\[
 |S_j|=1,
 \qquad
 |S_i|=q_i-2.
\]
By \eqref{eq211} in Lemma~\ref{lem23}, with $m=1$,
\[
 \begin{aligned}
 \#\{u\pmod Q:u\in\widehat{\mathcal G}\}
 &=|S_j|\prod_{\substack{i\le k\\i\ne j}}|S_i|\\
 &=\prod_{\substack{i\le k\\i\ne j}}(q_i-2)\\
 &=\frac Qq
 \prod_{\substack{i\le k\\i\ne j}}
 \left(1-\frac2{q_i}\right).
 \end{aligned}
\]
For $h=2$, the choice of the primes gives
\[
 q_i^{-1/2}<2^{-i-10}.
\]
Thus
\[
 \frac2{q_i}<2^{1-2(i+10)},
 \qquad
 \sum_{i\ge1}\frac2{q_i}<\frac12.
\]
Using the inequality
\[
 \prod_i(1-x_i)\ge1-\sum_i x_i
\]
for $0\le x_i\le1$, with $x_i=2/q_i$, we obtain
\[
 \prod_i\left(1-\frac2{q_i}\right)>\frac12.
\]
Let
\[
 \mathcal B_r=[rQ,(r+1)Q-1]\cap\mathbb Z,
 \qquad
 \mathcal M(I')=\{r\in\mathbb Z:\mathcal B_r\subseteq I'\}.
\]
Then
\[
 \#\bigl(\widehat{\mathcal G}\cap\mathcal B_r\bigr)
 =\frac Qq\prod_{\substack{i\le k\\i\ne j}}
 \left(1-\frac2{q_i}\right)
 >\frac Q{2q}
 \qquad(r\in\mathcal M(I')).
\]
Moreover,
\[
 \left|I'\setminus\bigcup_{r\in\mathcal M(I')}\mathcal B_r\right|<2Q,
 \qquad
 |\mathcal M(I')|Q\ge |I'|-2Q,
\]
and hence
\[
 |\mathcal M(I')|\ge\frac{|I'|-2Q}{Q}.
\]
Since
\[
 I'\subseteq I,
 \qquad
 \mathcal G=I\cap\widehat{\mathcal G},
 \qquad
 I'\cap\widehat{\mathcal G}=I'\cap\mathcal G,
\]
we obtain
\[
 \begin{aligned}
 |I'\cap\mathcal G|
 &\ge\sum_{r\in\mathcal M(I')}
 \#\bigl(\widehat{\mathcal G}\cap\mathcal B_r\bigr)\\
 &> |\mathcal M(I')|\frac Q{2q}\\
 &\ge\frac{|I'|-2Q}{Q}\cdot\frac Q{2q}\\
 &=\frac{|I'|}{2q}-\frac Qq.
 \end{aligned}
\]
In the case $p=1$, \eqref{eq356} gives
\[
 |I'|=|J_1|\ge\frac{\zeta}{2}T.
\]
Put
\[
 c_1=\frac{\zeta}{2},\qquad c_0=\frac{c_1}{8}=\frac{\zeta}{16}.
\]
Since $Q\le\eta T$ and $\eta\le\zeta/100<c_1/4$,
\[
 |I'\cap\mathcal G|
 \ge\left(\frac{c_1}{2}-\eta\right)\frac Tq
 \ge\frac{c_1}{4}\frac Tq.
\]
Then
\[
 |I'\cap\mathcal G|
 \ge2c_0\frac Tq.
\]

Let
\[
 K_t=|\mathcal G\setminus\mathcal G_0|.
\]
If $k=N$, or if $k<N$ and $q_{k+1}>2L$, then $\mathcal G_0=\mathcal G$, and therefore
\[
 K_t=0.
\]
It remains to consider the case
\[
 k<N,
 \qquad
 \ell=q_{k+1}\le2L.
\]
If $n\in\mathcal G\setminus\mathcal G_0$, then
\[
 n\equiv t\pmod q,
 \qquad
 n-a\equiv e_{k+1}\pmod\ell.
\]
Since $(q,\ell)=1$, the two congruences determining
$\mathcal G\setminus\mathcal G_0$ specify one residue class modulo $q\ell$. Hence
\[
 K_t\le \frac{|I|}{q\ell}+1.
\]
By $I=J_0$ and \eqref{eq355},
\[
 |I|\le\zeta T+1,
 \qquad
 K_t\le\frac{\zeta T}{q\ell}+2.
\]
By the definitions of $X_t,Y_t$,
\[
 |X_t|\le K_t,
 \qquad
 |Y_t|\le K_t,
 \qquad
 |X_t+Y_t|
 \le|X_t||Y_t|
 \le K_t^2.
\]
Consequently,
\[
 \begin{aligned}
 \frac{K_t+K_t^2}{T/q}
 &\le
 \frac{5\zeta}{\ell}
 +\frac{\zeta^2T}{q\ell^2}
 +\frac{6q}{T}.
 \end{aligned}
\]
By the maximality of $k$ and \eqref{eq21},
\[
 Q\ell>\eta T,
 \qquad
 \ell>C_*Q^4,
 \qquad
 q\le Q\le\eta T,
\]
and therefore
\[
 \frac{\zeta^2T}{q\ell^2}
 <\frac{\zeta^2}{C_*\eta qQ^3},
 \qquad
 \frac qT\le\eta.
\]
Using \eqref{eq39}, $h\ge2$, and $c_0=\zeta/16$,
\[
 q_0>1000h^3\ge8000,
 \qquad
 \eta\le\frac{\zeta}{1000},
 \qquad
 C_*>\frac{96}{\zeta\eta q_0^4}.
\]
Since $q,Q,\ell>q_0$ and $0<\zeta\le1/2$,
\[
 \begin{aligned}
 \frac{K_t+K_t^2}{T/q}
 &<\frac{5\zeta}{8000}
   +\frac{\zeta^3}{96}
   +\frac{6\zeta}{1000}\\
 &<\frac{\zeta}{32}
 =\frac{c_0}{2}.
 \end{aligned}
\]
Since
\[
 \mathcal P_0=I'\cap\mathcal G_0
\]
and $\mathcal P_a(L)=\mathcal P_0\setminus(X_t+Y_t)$,
\[
 \begin{aligned}
 |\mathcal P_a(L)|
 &\ge|I'\cap\mathcal G|-K_t-K_t^2\\
 &\ge2c_0\frac Tq-\frac{c_0}{2}\frac Tq\\
 &\ge c_0\frac Tq.
 \end{aligned}
\]
Therefore
\[
 |\mathcal P_a(L)|\gg T/q.
\]
\end{proof}

\begin{lemma}\label{lem37}
There exist constants $L_*,c,C_1>0$, depending only on $h,\theta$ and the fixed parameters, with the following property. Whenever $L\ge L_*$ and \eqref{eq310}, \eqref{eq21}, \eqref{eq22}, and \eqref{eq311} hold, there exists a set $A_1\supseteq A_0$ such that
\begin{align}
 A_1&\subseteq[0,L],\qquad
 A_1\setminus A_0\subseteq[\delta L,L],\label{eq363}\\
 cL^\theta&\le|A_1\setminus A_0|\le C_1L^\theta,\label{eq364}\\
 [\lambda L,\mu L]&\subseteq hA_1.\label{eq365}
\end{align}
For every $x\in A_1\setminus A_0$,
\begin{equation}\label{eq366}
\begin{aligned}
 x&\not\equiv e_i\pmod{q_i}\quad(1\le i\le k),\\
 x&\not\equiv e_{k+1}\pmod{q_{k+1}}
 \quad(k<N,\ q_{k+1}\le2L).
\end{aligned}
\end{equation}
Moreover, there exists a set $\mathcal P_a(L)$ such that
\begin{align}
 \mathcal P_a(L)&\subseteq[0,WL]\cap\bigl(hA_1\setminus h(A_1\setminus\{a\})\bigr),\label{eq367}\\
 |\mathcal P_a(L)|&\ge cT/q_j.\label{eq368}
\end{align}
\end{lemma}

\begin{proof}
By Lemma~\ref{lem34}, choose $\mathcal E(L)$ satisfying \eqref{eq320}--\eqref{eq325}. We construct
\[
 D_{\rm o}\subseteq[\delta L,\rho L],\qquad H_i^{\rm o}\subseteq B_H(L)\quad(0\le i<p),
\]
so that
\begin{equation}\label{eq369}
 |D_{\rm o}|+\sum_{i=0}^{p-1}|H_i^{\rm o}|=O(L^\theta).
\end{equation}
For every $x\in D_{\rm o}\cup\bigcup_iH_i^{\rm o}$,
\begin{equation}\label{eq370}
\begin{aligned}
 x&\not\equiv e_\nu\pmod{q_\nu}\quad(1\le\nu\le k),\\
 x&\not\equiv e_{k+1}\pmod{q_{k+1}}
 \quad(k<N,\ q_{k+1}\le2L).
\end{aligned}
\end{equation}
In addition, every $x\in H_i^{\rm o}$ satisfies
\begin{equation}\label{eq371}
 x\equiv r\pmod q,
\end{equation}
and
\begin{equation}\label{eq372}
 J\setminus(t+q\mathbb Z)\subseteq D_{\rm o}+H_0^{\rm o}+\cdots+H_{p-1}^{\rm o}.
\end{equation}

Since $\alpha H=\omega/p$, sums of integer intervals are again integer intervals, and
\[
 \begin{aligned}
 \min\bigl([\delta L,\rho L]+pB_H(L)\bigr)
 &=\lceil\delta L\rceil+p\lceil(\alpha H-\rho)L\rceil\\
 &\le(\omega-p\rho+\delta)L+h,\\
 \max\bigl([\delta L,\rho L]+pB_H(L)\bigr)
 &=\lfloor\rho L\rfloor+p\lfloor(\alpha H+\rho)L\rfloor\\
 &\ge(\omega+(p+1)\rho)L-h.
 \end{aligned}
\]
By $\eps=\rho/100$ and $\delta<\eps/(10h)$,
\[
 p\rho-\delta-\eps>0,\qquad
 (p+1)\rho-\eps>0,\qquad
 \rho-\delta>0.
\]
Choose the fixed constant
\[
 c=c_{\rm J}=\frac14\min\{p\rho-\delta-\eps,\,
 (p+1)\rho-\eps,\,\rho-\delta,\,2\rho\}>0.
\]
For all sufficiently large $L$,
\[
 \begin{aligned}
 (\omega-p\rho+\delta)L+h
 &\le(\omega-\eps-c)L,\\
 (\omega+(p+1)\rho)L-h
 &\ge(\omega+\eps+c)L,\\
 (\rho-\delta)L-1&\ge cL,\\
 2\rho L-1&\ge cL.
 \end{aligned}
\]
Hence
\[
 J+[-cL,cL]\subseteq[\delta L,\rho L]+pB_H(L),
 \qquad
 |[\delta L,\rho L]|,\ |B_H(L)|\ge cL.
\]

Put
\[
 S_0=\mathbb Z_{q}\setminus\{a\},\qquad
 S_i=\{r\}\quad(1\le i\le p).
\]
Since $r-a\equiv hB_0\not\equiv0\pmod q$ and
$t\equiv a+pr\pmod q$,
\[
 a\notin S_i\quad(0\le i\le p),\qquad
 S_0+\cdots+S_p
 =(\mathbb Z_{q}\setminus\{a\})+\{pr\}
 =\mathbb Z_{q}\setminus\{t\}.
\]
Therefore, by Lemma~\ref{lem27} with $d=h$, there exist sets $R_0,\ldots,R_p$ modulo $Q$ such that
\[
 \begin{aligned}
 \Omega:=(R_0+\cdots+R_p)\bmod Q
 &=\{u\bmod Q:u\not\equiv t\pmod q\},\\
 |R_0|&\ll q(Q/q)^{1/h},\qquad
 |R_i|\ll(Q/q)^{1/h}\quad(1\le i\le p).
 \end{aligned}
\]

From
\[
 J+[-cL,cL]\subseteq[\delta L,\rho L]+pB_H(L),\qquad
 |[\delta L,\rho L]|,\ |B_H(L)|\ge cL,\qquad
 Q\le\eta T\le\eta\sigma L,
\]
apply Lemma~\ref{lem28}\textup{(i)} with
\[
 m=h,\quad X=L,\quad U=J,\quad
 V_0=[\delta L,\rho L],\quad
 V_i=B_H(L)\ (1\le i\le p),
\]
and denote the resulting sets by
\[
 D_{\rm o}=D_0,\qquad H_{i-1}^{\rm o}=D_i\quad(1\le i\le p).
\]
By \eqref{eq227}, \eqref{eq24}, and \eqref{eq223}, respectively,
\[
 \begin{aligned}
 J\setminus(t+q\mathbb Z)
 &=J\cap\{n:n\bmod Q\in\Omega\}\\
 &\subseteq D_{\rm o}+H_0^{\rm o}+\cdots+H_{p-1}^{\rm o},\\
 x&\not\equiv e_\nu\pmod{q_\nu}
 \quad(x\in D_{\rm o}\cup\bigcup_iH_i^{\rm o},\ 1\le\nu\le k),\\
 H_i^{\rm o}&\equiv r\pmod q\quad(0\le i<p).
 \end{aligned}
\]
Moreover,
\[
 x\not\equiv e_{k+1}\pmod{q_{k+1}}\qquad
 \left(x\in D_{\rm o}\cup\bigcup_{i=0}^{p-1}H_i^{\rm o},\ k<N,\ q_{k+1}\le2L\right).
\]

Since $q^p\le L^{h\theta-1}$ and $\theta>1/h$,
\[
 \begin{aligned}
 |D_{\rm o}|
 &\ll (L/Q)^{1/h}q(Q/q)^{1/h}
 =q^{p/h}L^{1/h}\le L^\theta,\\
 |H_i^{\rm o}|
 &\ll (L/Q)^{1/h}(Q/q)^{1/h}
 =q^{-1/h}L^{1/h}\le L^\theta\quad(0\le i<p),\\
 |D_{\rm o}|+\sum_{i=0}^{p-1}|H_i^{\rm o}|
 &=O(L^\theta)\qquad(h=p+1\ \text{fixed}).
 \end{aligned}
\]

By Lemma~\ref{lem35}, choose
\[
 J_1\subseteq J_0\subseteq K,\qquad
 V\subseteq B_C(L),\qquad U_i\subseteq B_{O_i}(L)\ (0\le i<p),
\]
satisfying \eqref{eq347}--\eqref{eq350}; if $p\ge2$, they also satisfy \eqref{eq351}--\eqref{eq352}, while if $p=1$, they also satisfy \eqref{eq353}--\eqref{eq354}.

If $p\ge2$, construct $D_0,\ldots,D_{p-1}\subseteq B_H(L)$ such that, for every $0\le i<p$ and every $x\in D_i$,
\begin{equation}\label{eq373}
 x\equiv r\pmod q,
\end{equation}
and
\begin{equation}\label{eq374}
\begin{aligned}
 x&\not\equiv e_\nu\pmod{q_\nu}\quad(1\le\nu\le k),\\
 x&\not\equiv e_{k+1}\pmod{q_{k+1}}
 \quad(k<N,\ q_{k+1}\le2L).
\end{aligned}
\end{equation}
Furthermore,
\begin{align}
 \sum_{i=0}^{p-1}|D_i|&=O(T^{1/p}),\label{eq375}\\
 J_0\cap(t+q\mathbb Z)&\subseteq a+D_0+\cdots+D_{p-1}.\label{eq376}
\end{align}

Since $p\ge2$ and $r-a\equiv hB_0\not\equiv0\pmod q$, apply Lemma~\ref{lem27} with $d=p$ and $S_i=\{r\}$ for $0\le i<p$. We obtain
\[
 \begin{aligned}
 \Omega:=(R_0+\cdots+R_{p-1})\bmod Q
 &=\{u\bmod Q:u\equiv pr\pmod q\},\\
 |R_i|&\ll(Q/q)^{1/p},\qquad
 R_i\cap(e_\nu+q_\nu\mathbb Z)=\varnothing
 \quad(1\le\nu\le k).
 \end{aligned}
\]

Let $m_0=\operatorname{mid}(J_0)$ and $c_0=(m_0-a)/p$. Since $J_0\subseteq K$, $a\le\rho T/10$, $\alpha H=\omega/p$, and $T\le\sigma L$,
\[
 |c_0-\alpha HL|
 \le\frac{\eps L}{4p}+\frac{\rho T}{10p}
 \le\left(\frac{\rho}{400p}
     +\frac{\rho\sigma}{10p}\right)L<\frac\rho4L.
\]
By \eqref{eq355}, $\zeta\le1/2$, and sufficiently large $T$,
\[
 |J_0|\le\zeta T+1\le T.
\]
Take $c_1=1/(10p)$ and $C_2=1/p$, and set
\[
 V_i=[c_0-C_2T,c_0+C_2T]\cap\mathbb Z
 \qquad(0\le i<p),\qquad U=J_0-a.
\]
Since $\sigma<\eps/(100h)$,
\[
 C_2\sigma<\rho/4,\qquad
 |c_0-\alpha HL|+C_2T
 <\tfrac12\rho L,\qquad V_i\subseteq B_H(L).
\]
Also, since $\sum_{i=0}^{p-1}c_0=m_0-a$,
\[
 \begin{aligned}
 \min V_i&=\lceil c_0-C_2T\rceil\le c_0-C_2T+1,\\
 \max V_i&=\lfloor c_0+C_2T\rfloor\ge c_0+C_2T-1,\\
 |V_i|&\ge2C_2T-1=\frac{2T}{p}-1,
 \end{aligned}
\]
and
\[
 \begin{aligned}
 \min(V_0+\cdots+V_{p-1})
 &\le p(c_0-C_2T+1)=m_0-a-T+p,\\
 \max(V_0+\cdots+V_{p-1})
 &\ge p(c_0+C_2T-1)=m_0-a+T-p.
 \end{aligned}
\]
Since $p\ge2$, $c_1=1/(10p)$, and $|U|=|J_0|\le T$, for $T\ge3p$,
\[
 \frac{|U|}{2}+c_1T
 \le\left(\frac12+\frac1{10p}\right)T
 \le\frac{11}{20}T
 <T-p.
\]
Consequently,
\[
 \begin{gathered}
 |V_i|\ge\frac{2T}{p}-1\ge\frac{T}{p}\ge c_1T,
 \qquad |U|\le T,\\
 U+[-c_1T,c_1T]
 \subseteq[m_0-a-(T-p),m_0-a+(T-p)]\cap\mathbb Z
 \subseteq V_0+\cdots+V_{p-1}.
 \end{gathered}
\]
Since $Q\le\eta T$, apply Lemma~\ref{lem28}\textup{(i)} with $m=p$ and $X=T$. Then
\[
 \begin{aligned}
 D_i&\subseteq V_i\subseteq B_H(L),\qquad D_i\equiv r\pmod q,\\
 x&\not\equiv e_\nu\pmod{q_\nu}
 \quad(x\in D_i,\ 1\le\nu\le k),\\
 x&\not\equiv e_{k+1}\pmod{q_{k+1}}
 \quad(x\in D_i,\ k<N,\ q_{k+1}\le2L).
 \end{aligned}
\]
Moreover,
\[
 \begin{aligned}
 (J_0-a)\cap\{u:u\bmod Q\in\Omega\}
 &\subseteq D_0+\cdots+D_{p-1},\\
 |D_i|&\ll(T/Q)^{1/p}(Q/q)^{1/p}
 =(T/q)^{1/p}.
 \end{aligned}
\]
If $n\in J_0\cap(t+q\mathbb Z)$, then, since $t\equiv a+pr\pmod q$,
\[
 n-a\in(J_0-a)\cap\{u:u\bmod Q\in\Omega\},
 \qquad n\in a+D_0+\cdots+D_{p-1}.
\]
Thus
\[
 \sum_{i=0}^{p-1}|D_i|
 \ll p(T/q)^{1/p}=O(T^{1/p}).
\]

Put
\[
 \mathcal A^{(p)}=A_0\cup\mathcal E(L)\cup D_{\rm o}
 \cup\bigcup_iH_i^{\rm o}\cup V\cup\bigcup_iU_i\cup\bigcup_iD_i.
\]
Then
\begin{equation}\label{eq377}
 \cP_a(L)=J_1\cap(t+q\mathbb Z)
\end{equation}
satisfies
\begin{equation}\label{eq378}
 \cP_a(L)\subseteq
 h\mathcal A^{(p)}\setminus h(\mathcal A^{(p)}\setminus\{a\}),
\end{equation}
and
\begin{equation}\label{eq379}
 |\cP_a(L)|\ge |J_1|/q-2\gg T/q.
\end{equation}

Let
\[
 \mathcal C:=\mathcal A^{(p)}\setminus\mathcal E(L).
\]
From
\[
 \begin{gathered}
 A_0\subseteq[0,L/R]\subseteq[0,\tau L],\qquad
 D_{\rm o}\subseteq[\delta L,\rho L]\subseteq[0,\tau L],\\
 H_i^{\rm o}\subseteq B_H(L)\quad(0\le i<p),\qquad
 D_i\subseteq B_H(L)\quad(0\le i<p),\\
 V\subseteq B_C(L),\qquad
 U_i\subseteq B_{O_i}(L)\quad(0\le i<p),
 \end{gathered}
\]
we obtain
\[
 \begin{aligned}
 \mathcal C
 &\subseteq A_0\cup D_{\rm o}\cup\bigcup_{i=0}^{p-1}H_i^{\rm o}
 \cup V\cup\bigcup_{i=0}^{p-1}U_i\cup\bigcup_{i=0}^{p-1}D_i\\
 &\subseteq[0,\tau L]\cup B_H(L)\cup B_C(L)
 \cup\bigcup_{i=0}^{p-1}B_{O_i}(L)\\
 &= [0,\tau L]\cup\bigcup_{d\in\mathcal D}B_d(L).
 \end{aligned}
\]
By \eqref{eq371}, \eqref{eq351}, and \eqref{eq373},
\[
 \begin{aligned}
 x-a&\equiv hB_0\pmod q
 &&(x\in\mathcal C\cap B_H(L)),\\
 x-a&\equiv pB_0\pmod q
 &&\left(x\in\mathcal C\cap
 \left(B_C(L)\cup\bigcup_{i=0}^{p-1}B_{O_i}(L)\right)\right).
 \end{aligned}
\]
Hence
\[
 (x-a)\bmod q\in\{pB_0,hB_0\}\bmod q
 \subseteq[pB_0-2,hB_0]\bmod q
 =\mathcal R_q.
\]
Thus \eqref{eq323} and \eqref{eq324} hold. By \eqref{eq376},
\[
 \cP_a(L)\subseteq h\mathcal A^{(p)}.
\]
For $n\in\cP_a(L)$, \eqref{eq325} of Lemma~\ref{lem34} excludes every representation containing an element of $\mathcal E(L)$. Since
\[
 J_1\subseteq K\subseteq J,\qquad
 \mathcal C\subseteq[0,\tau L]\cup\bigcup_{d\in\mathcal D}B_d(L),
\]
\eqref{eq36} leaves only the two interval patterns
\[
 \{C,O_0,\ldots,O_{p-1}\},\qquad
 \{0,H,\ldots,H\}.
\]
Since
\[
 \mathcal C\cap B_C(L)\subseteq V,
 \qquad
 \mathcal C\cap B_{O_i}(L)\subseteq U_i\quad(0\le i<p),
\]
a representation of the first type would imply
\[
 n\in V+\sum_{i=0}^{p-1}U_i,
\]
contrary to $n\in J_1$ and \eqref{eq349}. A representation of the second type has the form
\[
 n=x_0+\sum_{i=1}^p x_i,\qquad
 x_0\in\mathcal C\cap[0,\tau L],\qquad
 x_i\in\mathcal C\cap B_H(L).
\]
Since $x_i\equiv r\pmod q$ for $1\le i\le p$,
\[
 \begin{aligned}
 x_0&\equiv n-pr\equiv t-pr\equiv a\pmod q,\\
 \mathcal A^{(p)}\cap(a+q\mathbb Z)&=\{a\},\\
 x_0&=a.
 \end{aligned}
\]
Indeed,
\[
 a=e_j,\qquad q=q_j,\qquad j\le k,
\]
\[
 \eqref{eq22}
 \Longrightarrow
 A_0\cap(a+q\mathbb Z)=\{a\},
\]
and
\[
 \mathcal A^{(p)}\setminus A_0
 \subseteq\{x:x\not\equiv e_j\pmod{q_j}\}
 \Longrightarrow
 (\mathcal A^{(p)}\setminus A_0)\cap(a+q\mathbb Z)=\varnothing.
\]
Therefore
\[
 \mathcal A^{(p)}\cap(a+q\mathbb Z)=\{a\}.
\]
Hence
\[
 \cP_a(L)\subseteq
 h\mathcal A^{(p)}\setminus h(\mathcal A^{(p)}\setminus\{a\}).
\]
Since $p\ge2$, \eqref{eq356} gives
\[
 |J_1|\ge2^{-p-1}\zeta T.
\]
Also $q\le Q\le\eta T$, and hence
\[
 2\le 2\eta\,\frac{T}{q}.
\]
Using
\[
 \eta\le\frac{2^{-p-2}\zeta}{p+1}
 \le2^{-p-3}\zeta,
\]
we obtain
\[
 \begin{aligned}
 |\cP_a(L)|
 &=|J_1\cap(t+q\mathbb Z)|\\
 &\ge |J_1|/q-2\\
 &\ge\left(2^{-p-1}\zeta-2\eta\right)\frac{T}{q}\\
 &\ge2^{-p-2}\zeta\,\frac{T}{q}.
 \end{aligned}
\]

If $p=1$, put $I=J_0$ and $I'=J_1$. By Lemma~\ref{lem36}, choose
\[
 D_{\rm r}\subseteq B_H(L),\qquad
 X\subseteq B_C(L),\qquad
 Y\subseteq B_{O_0}(L),
\]
and put
\[
 \mathcal A^{(1)}:=A_0\cup\mathcal E(L)\cup D_{\rm o}\cup H_0^{\rm o}
 \cup V\cup U_0\cup D_{\rm r}\cup X\cup Y,
\]
so that \eqref{eq358}--\eqref{eq360} and \eqref{eq361}--\eqref{eq362} hold simultaneously.

Set
\[
 \mathcal A^*=A_0\cup\mathcal E(L)\cup D_{\rm o}\cup\bigcup_iH_i^{\rm o}
 \cup V\cup\bigcup_iU_i.
\]
If $p\ge2$, let
\[
 \mathcal A^{**}=\mathcal A^*\cup\bigcup_{i=0}^{p-1}D_i.
\]
If $p=1$, let
\[
 \mathcal A^{**}=\mathcal A^*\cup D_{\rm r}\cup X\cup Y.
\]
By \eqref{eq322}, \eqref{eq372}, and \eqref{eq348},
\[
 \begin{aligned}
 \relax[\lambda L,\mu L]\setminus J
 &\subseteq h\mathcal E(L)\subseteq h\mathcal A^{**},\\
 J\setminus(t+q\mathbb Z)
 &\subseteq D_{\rm o}+\sum_{i=0}^{p-1}H_i^{\rm o}
 \subseteq h\mathcal A^{**},\\
 (J\setminus J_0)\cap(t+q\mathbb Z)
 &\subseteq V+\sum_{i=0}^{p-1}U_i
 \subseteq h\mathcal A^{**}.
 \end{aligned}
\]
Inside $J_0\cap(t+q\mathbb Z)$, if $p\ge2$, then
\[
 J_0\cap(t+q\mathbb Z)
 \subseteq a+\sum_{i=0}^{p-1}D_i\subseteq h\mathcal A^{**};
\]
if $p=1$, then
\[
 J_0\cap(t+q\mathbb Z)
 \subseteq(a+D_{\rm r})\cup(X+Y)\subseteq2\mathcal A^{**}.
\]
Therefore, decomposing $[\lambda L,\mu L]$ into $[\lambda L,\mu L]\setminus J$ and $J$,
\[
 [\lambda L,\mu L]\subseteq h\mathcal A^{**}.
\]

If $p\ge2$, then by $q^p\le L^{h\theta-1}$, $T\asymp L^\kappa$, $\kappa=p\theta$, and $\theta>1/h$,
\[
\begin{aligned}
 |\mathcal A^{**}\setminus A_0|
 &\ll q^{p/h}L^{1/h}+T^{1/p}+L/T\\
 &\ll L^\theta+L^{\kappa/p}+L^{1-\kappa}\\
 &\le3L^\theta
 \qquad(1-\kappa=1-p\theta\le\theta)\\
 &\ll L^\theta.
\end{aligned}
\]
If $p=1$, then by \eqref{eq311}, \eqref{eq354}, and \eqref{eq358},
\[
\begin{aligned}
 |\mathcal A^{**}\setminus A_0|
 &\ll q^{1/2}L^{1/2}+T+2^kL/T\\
 &\ll L^\theta+L^\theta+L^\theta\\
 &\ll L^\theta.
\end{aligned}
\]
Since
\[
 \delta<\tau<\alpha d-\rho,
 \qquad
 \alpha d+\rho\le\alpha H+\rho=\frac\omega p+\rho<1
 \quad(d\in\mathcal D),
\]
we have
\[
 \mathcal E(L)\subseteq[\delta L,L],\qquad
 [\delta L,\rho L]\cup\bigcup_{d\in\mathcal D}B_d(L)
 \subseteq[\delta L,L].
\]
Together with \eqref{eq321}, \eqref{eq370}, \eqref{eq350}, and either \eqref{eq374} when $p\ge2$ or \eqref{eq359} when $p=1$, this yields
\[
 \begin{aligned}
 \mathcal A^{**}\setminus A_0&\subseteq[\delta L,L],\\
 x&\not\equiv e_i\pmod{q_i}
 \quad(x\in\mathcal A^{**}\setminus A_0,\ 1\le i\le k),\\
 x&\not\equiv e_{k+1}\pmod{q_{k+1}}
 \quad(x\in\mathcal A^{**}\setminus A_0,\ k<N,\ q_{k+1}\le2L).
 \end{aligned}
\]

If $p\ge2$, then $\mathcal A^{**}=\mathcal A^{(p)}$, and by \eqref{eq377}--\eqref{eq379},
\[
 \mathcal P_a(L)\subseteq J_1\subseteq K\subseteq[0,WL].
\]
If $p=1$, then $\mathcal A^{**}=\mathcal A^{(1)}$, and by \eqref{eq361} and \eqref{eq362},
\[
 \mathcal P_a(L)\subseteq I'=J_1\subseteq K\subseteq[0,WL].
\]
Hence there is a constant $c_2>0$, independent of $L,q$, such that
\begin{equation}\label{eq380}
 \cP_a(L)\subseteq[0,WL]\cap
 \bigl(h\mathcal A^{**}\setminus h(\mathcal A^{**}\setminus\{a\})\bigr),
 \qquad |\cP_a(L)|\ge c_2T/q.
\end{equation}

Take
\[
 F=[(1-\rho)L,L]\cap\mathbb Z.
\]
Since $A_0\subseteq[0,L/R]$, $1/R<\tau=2\rho<1-\rho$, and $1-\rho>W$,
\[
 F\cap A_0=\varnothing,\qquad \min F>WL.
\]
For every $i\le k$,
\[
 \#\{x\in F:x\equiv e_i\pmod{q_i}\}
 \le\frac{|F|}{q_i}+1.
\]
Therefore
\[
 \#\left(F\setminus\bigcup_{i\le k}(e_i+q_i\mathbb Z)\right)
 \ge|F|\left(1-\sum_{i\le k}\frac1{q_i}\right)-k.
\]
By Lemma~\ref{lem21},
\[
 \sum_{i\le k}\frac1{q_i}<2^{-10},
 \qquad k=O(\sqrt{\log L})=o(L).
\]
If $k<N$ and $q_{k+1}\le2L$, then by Lemma~\ref{lem21} and $T\asymp L^\kappa\to\infty$,
\[
 q_{k+1}>C_*^{1/5}(\eta T)^{4/5}\to\infty.
\]
Hence in this case
\[
 \#(F\cap(e_{k+1}+q_{k+1}\mathbb Z))
 \le\frac{|F|}{q_{k+1}}+1=o(L).
\]
Define
\[
 \mathcal F=
 F\setminus\Biggl[
 \bigcup_{i\le k}(e_i+q_i\mathbb Z)
 \cup
 \begin{cases}
 e_{k+1}+q_{k+1}\mathbb Z,&k<N,\ q_{k+1}\le2L,\\
 \varnothing,&\text{otherwise}.
 \end{cases}
 \Biggr].
\]
Since $|F|=\rho L+O(1)$, $k=o(L)$, and
\[
 \sum_{i\le k}\frac1{q_i}<2^{-10},
 \qquad
 \mathbf 1_{\{k<N,\ q_{k+1}\le2L\}}
 \left(\frac{|F|}{q_{k+1}}+1\right)=o(L),
\]
there is a fixed constant $c_1>0$ such that
\[
 |\mathcal F|
 \ge |F|(1-2^{-10})-k-o(L)
 \ge c_1L.
\]
Since $|\mathcal A^{**}\setminus A_0|\ll L^\theta$, choose a fixed constant $C_0$ such that
\[
 M_L:=|\mathcal A^{**}\setminus A_0|\le C_0L^\theta.
\]
Choose $0<c\le c_2$ and $c<c_1/4$. If $M_L\ge\lceil cL^\theta\rceil$, set
\[
 F_0=\varnothing.
\]
If $M_L<\lceil cL^\theta\rceil$, then by $\theta\le1$, $c<c_1/4$, and sufficiently large $L$,
\[
\begin{aligned}
 |\mathcal F\setminus\mathcal A^{**}|
 &\ge c_1L-M_L\\
 &>c_1L-cL^\theta-1\\
 &\ge\frac{c_1}{2}L
 \ge\lceil cL^\theta\rceil-M_L.
\end{aligned}
\]
Thus choose
\[
 F_0\subseteq\mathcal F\setminus\mathcal A^{**},
 \qquad
 |F_0|=\lceil cL^\theta\rceil-M_L.
\]
In both cases $F_0\subseteq\mathcal F$; hence, by the definition of $\mathcal F$, $F_0$ satisfies \eqref{eq24}. Put
\[
 A_1=\mathcal A^{**}\cup F_0.
\]
Since $F_0\cap\mathcal A^{**}=\varnothing$ and $F_0\cap A_0=\varnothing$,
\[
 \begin{aligned}
 |A_1\setminus A_0|&=M_L+|F_0|,\\
 cL^\theta&\le |A_1\setminus A_0|
 \le(C_0+c+1)L^\theta.
 \end{aligned}
\]
Increasing the fixed constant $C_1$ gives \eqref{eq364}. If $F_0\ne\varnothing$, then from $F_0\subseteq F$ and $\mathcal P_a(L)\subseteq[0,WL]$,
\[
 \min F_0\ge(1-\rho)L>WL\ge\max\mathcal P_a(L).
\]
Since $A_1\subseteq\mathbb N_0$, for every $1\le r\le h$ and every $n\in rF_0+(h-r)A_1$,
\[
 n\ge\min F_0>\max\mathcal P_a(L).
\]
The same conclusion below is immediate when $F_0=\varnothing$. Therefore
\begin{equation}\label{eq381}
 \mathcal P_a(L)\cap\bigcup_{r=1}^h\left(rF_0+(h-r)A_1\right)=\varnothing.
\end{equation}
By \eqref{eq380}, \eqref{eq381}, and $h\mathcal A^{**}\subseteq hA_1$,
\[
 \begin{aligned}
 \mathcal P_a(L)
 &\subseteq[0,WL]\cap
 \bigl(hA_1\setminus h(A_1\setminus\{a\})\bigr),\\
 |\mathcal P_a(L)|&\ge c_2T/q\ge cT/q_j.
 \end{aligned}
\]
Finally, from
\[
 A_0\subseteq[0,L/R],\qquad
 \mathcal A^{**}\setminus A_0\subseteq[\delta L,L],\qquad
 F_0\subseteq[(1-\rho)L,L]\subseteq[\delta L,L],
\]
and
\[
 [\lambda L,\mu L]\subseteq h\mathcal A^{**}\subseteq hA_1,
\]
we obtain
\[
 \begin{aligned}
 A_1&\subseteq[0,L/R]\cup[\delta L,L]\subseteq[0,L],\\
 A_1\setminus A_0&\subseteq[\delta L,L],\\
 [\lambda L,\mu L]&\subseteq hA_1.
 \end{aligned}
\]
The lower threshold for $L$ is uniform in the admissible data. Indeed,
\[
 k\le\left(\frac{2\log L}{h\log2}\right)^{1/2},\qquad
 \sum_{i\le N}q_i^{-1}<2^{-10},\qquad
 \prod_{i\le N}(1+q_i^{-1/h})<e^{2^{-10}},
\]
and, when $k<N$,
\[
 q_{k+1}>C_*^{1/5}(\eta T)^{4/5}.
\]
All interval margins above are fixed positive multiples of $L$ or $T$;
$T\ge(\sigma/2)L^\kappa$ once $\sigma L^\kappa\ge2$.
For $p=1$, the only additional growth estimate follows uniformly from
\[
 \frac{2^kL/T}{L^\theta}
 \le\frac2\sigma
 \exp\left(\sqrt{\frac{2\log2}{h}\log L}\right)L^{1-2\theta}
 \longrightarrow0.
\]
The other estimates use only the displayed bounds and the admissibility
conditions \eqref{eq311}. Thus a single $L_*$ suffices for all the data in
the statement.
\end{proof}

\begin{proof}[Proof of Theorem~\ref{thm11}]
Let
\[
 A_0=\{0\},\qquad e_1=0.
\]
Choose a prime $q_1$ satisfying \eqref{eq21}, and put
\[
 Q_1=q_1.
\]
Since $\theta>1/h$ and $p=h-1\ge1$,
\[
 \kappa=p\theta>0,
 \qquad
 h\theta-1>0,
\]
we have
\[
 \lfloor\sigma L^\kappa\rfloor\longrightarrow\infty,
 \qquad L^{h\theta-1}\longrightarrow\infty
 \quad(L\longrightarrow\infty).
\]
Choose $L_1\ge\max\{1,L_*\}$, where $L_*$ is the uniform threshold in Lemma~\ref{lem37}, sufficiently large that
\begin{equation}\label{eq382}
 Q_1\le\eta T_1,
 \qquad
 e_1\le\rho T_1/10,
 \qquad
 q_1^p\le L_1^{h\theta-1},
 \qquad
 T_1=\lfloor\sigma L_1^\kappa\rfloor.
\end{equation}
Set
\[
 L_s=R^{s-1}L_1,
 \qquad
 T_s=\lfloor\sigma L_s^\kappa\rfloor
 \qquad(s\ge1).
\]
Since $R>1$ and $\kappa>0$,
\[
 L_{s+1}=RL_s,\qquad L_s\longrightarrow\infty,
 \qquad T_s\longrightarrow\infty.
\]
For $s\ge1$,
\[
 L_s\ge L_1,\qquad T_s\ge T_1,\qquad
 L_s^{h\theta-1}\ge L_1^{h\theta-1},
\]
and therefore \eqref{eq382} gives
\begin{equation}\label{eq383}
 Q_1\le\eta T_s,
 \qquad
 e_1\le\rho T_s/10,
 \qquad
 q_1^p\le L_s^{h\theta-1}
 \qquad(s\ge1).
\end{equation}
Thus the set in \eqref{eq23} is nonempty and
\[
 k\ge1.
\]
Moreover,
\[
 \begin{aligned}
 A_0&=\{e_1\}=\{0\}\subseteq[0,L_1/R],\\
 A_0\cap(e_1+q_1\mathbb Z)
 &=\{0\}\cap q_1\mathbb Z
 =\{0\}
 =\{e_1\}.
 \end{aligned}
\]

For $s\ge1$, suppose that
\[
 N_{s-1}:=|A_{s-1}|,
 \qquad
 A_{s-1}=\{e_1,\ldots,e_{N_{s-1}}\}\subseteq[0,L_s/R],
\]
and
\begin{equation}\label{eq384}
 A_{s-1}\cap(e_i+q_i\mathbb Z)=\{e_i\}
 \qquad(1\le i\le N_{s-1}).
\end{equation}
By \eqref{eq383}, define
\[
 k_s:=\max\{i\le N_{s-1}:Q_i\le\eta T_s\},
 \qquad
 Q^{(s)}:=Q_{k_s},
\]
and
\[
 v_2(s):=\max\{r\in\mathbb Z_{\ge0}:2^r\mid s\},
 \qquad
 i(s):=1+v_2(s).
\]
Put
\begin{equation}\label{eq385}
 a_s:=
 \begin{cases}
 e_{i(s)},&
 \begin{gathered}
 i(s)\le N_{s-1},\quad i(s)\le k_s,\\
 e_{i(s)}\le\rho T_s/10,\quad q_{i(s)}^p\le L_s^{h\theta-1},
 \end{gathered}\\[3mm]
 e_1,&\text{otherwise}.
 \end{cases}
\end{equation}
By \eqref{eq383},
\[
 1\le k_s,
 \qquad
 e_1\le\rho T_s/10,
 \qquad
 q_1^p\le L_s^{h\theta-1}.
\]
Thus, in either case of \eqref{eq385}, with
\[
 a:=a_s,
 \qquad k:=k_s,
 \qquad Q:=Q^{(s)},
\]
the conditions \eqref{eq311} are satisfied. Lemma~\ref{lem37} therefore gives a set $A_s\supseteq A_{s-1}$ such that
\begin{align}
 A_s&\subseteq[0,L_s],\label{eq386}\\
 A_s\setminus A_{s-1}&\subseteq[\delta L_s,L_s],\label{eq387}\\
 cL_s^\theta
 &\le|A_s\setminus A_{s-1}|\le C_1L_s^\theta,\label{eq388}\\
 [\lambda L_s,\mu L_s]&\subseteq hA_s.\label{eq389}
\end{align}

By \eqref{eq383}, $\kappa\le1$, $L_s\ge1$, and $\eta\sigma<1$,
\[
 Q_1\le\eta T_s\le\eta\sigma L_s^\kappa\le L_s.
\]
Apply Lemma~\ref{lem21} with
\[
 A_0=A_{s-1},\qquad X=A_s\setminus A_{s-1},\qquad L=L_s.
\]
Using \eqref{eq384}, \eqref{eq386}, \eqref{eq366}, and
\[
 A_s\setminus A_{s-1}\subseteq[0,L_s]\setminus A_{s-1},
\]
we obtain
\[
 A_s\cap(e_i+q_i\mathbb Z)=\{e_i\}
 \qquad(1\le i\le N_{s-1}).
\]

Let
\[
 M_s:=|A_s\setminus A_{s-1}|,
 \qquad
 A_s\setminus A_{s-1}=\{e_{N_{s-1}+1},\ldots,e_{N_{s-1}+M_s}\}.
\]
For
\[
 N_{s-1}<r\le N_{s-1}+M_s,
\]
choose a prime $q_r$ such that
\begin{equation}\label{eq390}
 q_r>q_0,
 \qquad
 q_r^{-1/h}<2^{-r-10},
 \qquad
 q_r>C_*Q_{r-1}^4,
 \qquad
 q_r>RL_s.
\end{equation}
The infinitude of primes allows these choices successively for every $N_{s-1}<r\le N_{s-1}+M_s$. Applying Lemma~\ref{lem21} with
\[
 B=A_s,\qquad e=e_r,\qquad q=q_r>RL_s,\qquad L=L_s,
\]
and using \eqref{eq386} and \eqref{eq390}, we obtain
\[
 A_s\cap(e_r+q_r\mathbb Z)=\{e_r\}
 \qquad(N_{s-1}<r\le N_{s-1}+M_s).
\]
Consequently,
\[
 |A_s|=N_{s-1}+M_s,
 \qquad
 A_s\cap(e_i+q_i\mathbb Z)=\{e_i\}
 \quad(1\le i\le |A_s|).
\]
Since $L_{s+1}/R=L_s$, \eqref{eq386} also gives
\[
 A_s\subseteq[0,L_{s+1}/R].
\]
From
\[
 A_0\subseteq A_1\subseteq[0,L_1],\qquad
 A_1\subseteq[0,L_2/R],\qquad
 A_1\cap(e_i+q_i\mathbb Z)=\{e_i\}
 \quad(1\le i\le|A_1|),
\]
together with \eqref{eq386} and \eqref{eq390}, induction yields, for every $s\ge1$,
\[
 \begin{gathered}
 A_{s-1}\subseteq A_s\subseteq[0,L_s],
 \qquad
 A_s\subseteq[0,L_{s+1}/R],\\
 A_s\cap(e_i+q_i\mathbb Z)=\{e_i\}
 \qquad(1\le i\le|A_s|).
 \end{gathered}
\]
Therefore
\[
 A_0\subseteq A_1\subseteq A_2\subseteq\cdots,
 \qquad
 A:=\bigcup_{s\ge0}A_s.
\]

By \eqref{eq388},
\[
 |A_s|
 =1+\sum_{t=1}^s|A_t\setminus A_{t-1}|
 \ge1+c\sum_{t=1}^sL_t^\theta.
\]
Since $R>1$,
\[
 L_t^\theta=L_1^\theta R^{(t-1)\theta},
 \qquad
 \sum_{t=1}^sL_t^\theta\longrightarrow\infty,
 \qquad
 |A_s|\longrightarrow\infty.
\]
Hence, for every fixed $i$, there exists $s_i^{(0)}\in\mathbb N$ such that
\[
 e_i\in A_s\qquad(s\ge s_i^{(0)}),
\]
and the corresponding $q_i,Q_i$ remain fixed thereafter.

Fix $i$. Since
\[
 T_s\longrightarrow\infty,
 \qquad
 L_s^{h\theta-1}\longrightarrow\infty,
\]
we may choose $S_i\ge s_i^{(0)}+1$ such that, for all $s\ge S_i$,
\begin{equation}\label{eq391}
 e_i\in A_{s-1},
 \qquad
 Q_i\le\eta T_s,
 \qquad
 e_i\le\rho T_s/10,
 \qquad
 q_i^p\le L_s^{h\theta-1}.
\end{equation}
From $e_i\in A_{s-1}$ and $Q_i\le\eta T_s$, the definition of $k_s$ gives
\[
 i\le k_s\qquad(s\ge S_i).
\]

By the definition of $v_2(s)$,
\[
 s=2^{v_2(s)}(2m+1)\quad\text{for some }m\in\mathbb Z_{\ge0}.
\]
Since $i(s)=1+v_2(s)$, the equality $i(s)=i$ is equivalent to $v_2(s)=i-1$. Hence
\begin{equation}\label{eq392}
 i(s)=i
 \iff
 s=2^{i-1}(2m+1)
 \quad(m\in\mathbb Z_{\ge0}).
\end{equation}
The difference between two consecutive such indices is
\[
 2^{i-1}\bigl(2(m+1)+1\bigr)
 -2^{i-1}(2m+1)
 =2^i.
\]
Combining \eqref{eq385}, \eqref{eq391}, and \eqref{eq392}, the set
\[
 \mathcal S_i:=\{s\ge S_i:a_s=e_i\}
\]
is infinite. If
\[
 s<t,\qquad s,t\in\mathcal S_i,
 \qquad
 \mathcal S_i\cap(s,t)=\varnothing,
\]
then
\begin{equation}\label{eq393}
 t-s\le G_i,
 \qquad
 G_i:=2^i.
\end{equation}
For $i\ge2$, this follows directly from \eqref{eq392}; for $i=1$, \eqref{eq385} may also give $a_s=e_1$ at additional indices $s$, which can only decrease the gap.

On the other hand, by \eqref{eq389} and $A_s\subseteq A$,
\[
 [\lambda L_s,\mu L_s]\subseteq hA_s\subseteq hA
 \qquad(s\ge1).
\]
By $\lambda R<\mu$ in \eqref{eq38},
\[
 L_{s+1}=RL_s,
 \qquad
 \lambda L_{s+1}
 =\lambda RL_s
 <\mu L_s.
\]
Therefore
\[
 \lceil\lambda L_{s+1}\rceil\le\lfloor\mu L_s\rfloor+1.
\]
Since $\lambda L_s\to\infty$,
\[
 [\lambda L_1,\infty)
 \subseteq
 \bigcup_{s\ge1}[\lambda L_s,\mu L_s]
 \subseteq hA.
\]
Thus $A$ is an asymptotic basis of order $h$.

To estimate $A(x)$ from above, let $x\ge\delta L_1$ and choose the unique $s$ such that
\begin{equation}\label{eq394}
 \delta L_s\le x<\delta L_{s+1}.
\end{equation}
If $t\ge s+1$, then by \eqref{eq387} and $L_t\ge L_{s+1}$,
\[
 A_t\setminus A_{t-1}\subseteq[\delta L_t,L_t],
 \qquad
 \delta L_t\ge\delta L_{s+1}>x.
\]
Hence
\[
 A(x)=|A_s\cap[0,x]|\le|A_s|.
\]
By \eqref{eq388},
\[
\begin{aligned}
 A(x)
 &\le1+\sum_{t=1}^s|A_t\setminus A_{t-1}|\\
 &\le1+C_1\sum_{t=1}^sL_t^\theta.
\end{aligned}
\]
Since $L_t=L_sR^{t-s}$,
\[
 \begin{aligned}
 \sum_{t=1}^sL_t^\theta
 &=L_s^\theta\sum_{r=0}^{s-1}R^{-r\theta}\\
 &\le\frac{L_s^\theta}{1-R^{-\theta}}.
 \end{aligned}
\]
From $\delta L_s\le x$,
\[
 L_s\le\delta^{-1}x.
\]
Therefore
\[
\begin{aligned}
 A(x)
 &\le1+\frac{C_1}{1-R^{-\theta}}L_s^\theta\\
 &\le1+\frac{C_1\delta^{-\theta}}{1-R^{-\theta}}x^\theta\\
 &\ll x^\theta.
\end{aligned}
\]

For $x\ge L_1$, choose the unique $s$ such that
\begin{equation}\label{eq395}
 L_s\le x<L_{s+1}=RL_s.
\end{equation}
By \eqref{eq387},
\[
 A_s\setminus A_{s-1}\subseteq[\delta L_s,L_s]\subseteq[0,x].
\]
Thus, by \eqref{eq388},
\[
 A(x)
 \ge|A_s\setminus A_{s-1}|
 \ge cL_s^\theta.
\]
Since $x<RL_s$,
\[
 L_s>R^{-1}x,
\]
and hence
\[
 A(x)
 \ge cL_s^\theta
 >cR^{-\theta}x^\theta.
\]
Combining the upper and lower bounds gives
\begin{equation}\label{eq396}
 A(x)\asymp x^\theta.
\end{equation}

Fix $e_i\in A$ and $s\in\mathcal S_i$. Since $a_s=e_i$, \eqref{eq367} of Lemma~\ref{lem37} gives
\[
 \mathcal P_{e_i}(L_s)
 \subseteq[0,WL_s]\cap
 \bigl(hA_s\setminus h(A_s\setminus\{e_i\})\bigr),
\]
and
\begin{equation}\label{eq397}
 |\mathcal P_{e_i}(L_s)|\ge cT_s/q_i.
\end{equation}

For $t\ge s+1$, \eqref{eq387} gives
\[
 A_t\setminus A_{t-1}\subseteq[\delta L_t,L_t]
 \subseteq[\delta L_{s+1},\infty).
\]
Taking the union over all $t\ge s+1$ and using $\delta R>W$ from \eqref{eq38},
\[
 A\setminus A_s\subseteq[\delta L_{s+1},\infty),
 \qquad
 \delta L_{s+1}
 =\delta RL_s
 >WL_s.
\]
If $n\in\mathcal P_{e_i}(L_s)$, then
\[
 n\le WL_s<\delta L_{s+1}.
\]
If an $h$-term representation $n=y_1+\cdots+y_h$ contains some $y_j\in A\setminus A_s$, then, since $A\subseteq\Nzero$,
\[
 n=\sum_{\nu=1}^h y_\nu
 \ge y_j\ge\delta L_{s+1}>n,
\]
a contradiction. Thus every $h$-term representation of $n$ uses only elements of $A_s$. Since
\[
 n\notin h(A_s\setminus\{e_i\}),\qquad
 n\in hA_s\subseteq hA,
\]
it follows that
\[
 n\notin h(A\setminus\{e_i\}),\qquad
 n\in hA.
\]
Therefore
\begin{equation}\label{eq398}
 \mathcal P_{e_i}(L_s)\subseteq E_{e_i}.
\end{equation}

For fixed $i$, write the infinite set $\mathcal S_i$ as
\[
 \mathcal S_i=\{s_{i,1}<s_{i,2}<s_{i,3}<\cdots\}.
\]
By \eqref{eq393},
\[
 s_{i,r+1}-s_{i,r}\le G_i=2^i
 \qquad(r\ge1).
\]
For all sufficiently large $x$, the set
\[
 \{s\in\mathcal S_i:WL_s\le x\}
\]
is nonempty and finite. Define
\[
 s:=\max\{u\in\mathcal S_i:WL_u\le x\},
 \qquad
 t:=\min\{u\in\mathcal S_i:u>s\}.
\]
Since $L_u$ is strictly increasing in $u$ and $t-s\le G_i$,
\[
 WL_s\le x<WL_t,\qquad t-s\le G_i.
\]
Hence
\[
 x<WL_t=WR^{t-s}L_s\le WR^{G_i}L_s,
\]
and therefore
\begin{equation}\label{eq399}
 L_s>\frac{x}{WR^{G_i}}.
\end{equation}

Using $WL_s\le x$, \eqref{eq398}, and \eqref{eq397}, we obtain
\[
 \begin{aligned}
 \mathcal P_{e_i}(L_s)
 &\subseteq E_{e_i}\cap[0,WL_s]
 \subseteq E_{e_i}\cap[0,x],\\
 E_{e_i}(x)
 &\ge|\mathcal P_{e_i}(L_s)|
 \ge\frac{cT_s}{q_i}.
 \end{aligned}
\]
By \eqref{eq399}, $L_s\to\infty$ as $x\to\infty$, and hence $s\to\infty$. Thus, for all sufficiently large $x$,
\[
 \sigma L_s^\kappa\ge2,\qquad
 T_s
 \ge\sigma L_s^\kappa-1
 \ge\frac{\sigma}{2}L_s^\kappa.
\]
Therefore, by \eqref{eq399},
\[
\begin{aligned}
 E_{e_i}(x)
 &\ge\frac{c\sigma}{2q_i}L_s^\kappa\\
 &>\frac{c\sigma}{2q_iW^\kappa R^{G_i\kappa}}x^\kappa.
\end{aligned}
\]
Thus, for fixed $e_i$,
\begin{equation}\label{eq3100}
 E_{e_i}(x)\gg_i x^\kappa
 =x^{\theta(h-1)}.
\end{equation}
Since $A=\{e_1,e_2,\ldots\}$, \eqref{eq3100} gives, for every $a=e_i\in A$,
\[
 E_a(x)\gg_a x^{\theta(h-1)}.
\]

By \eqref{eq396}, there is a fixed constant $C>0$ such that, for all sufficiently large $x$,
\[
 A(x)\le Cx^\theta.
\]
Hence
\[
 A(x)^p
 \le C^px^{p\theta}
 =C^px^\kappa.
\]
By \eqref{eq3100}, for every fixed $e_i$ there is a constant $c_i>0$ such that, for all sufficiently large $x$,
\[
 E_{e_i}(x)\ge c_ix^\kappa
 \ge\frac{c_i}{C^p}A(x)^p,
 \qquad p=h-1.
\]
Thus $E_{e_i}(x)\gg_i A(x)^{h-1}$ for every $e_i\in A$, and hence $A$ is strongly minimal.

\end{proof}


\begin{thebibliography}{99}\small

\bibitem{ChenChen2011}
F. J. Chen and Y. G. Chen,
\emph{On minimal asymptotic bases},
European J. Combin. \textbf{32} (2011), 1329--1335.

\bibitem{Chen2025Strongly}
S.-Q. Chen,
\emph{On strongly asymptotic bases of order $h$},
Ramanujan J. \textbf{66} (2025), Article~78.

\bibitem{DeshouillersFouvry1976}
J.-M. Deshouillers and E. Fouvry,
\emph{On additive bases (II)},
J. London Math. Soc. (2) \textbf{14} (1976), 413--422.

\bibitem{ErdosNathanson1977}
P. Erd\H{o}s and M. B. Nathanson,
\emph{Nonbases of density zero not contained in maximal nonbases},
J. London Math. Soc. (2) \textbf{15} (1977), 403--405.

\bibitem{ErdosNathanson1980}
P. Erd\H{o}s and M. B. Nathanson,
\emph{Minimal asymptotic bases for the natural numbers},
J. Number Theory \textbf{12} (1980), 154--159.

\bibitem{JiaNathanson1989}
X. D. Jia and M. B. Nathanson,
\emph{A simple construction of minimal asymptotic bases},
Acta Arith. \textbf{52} (1989), 95--101.

\bibitem{Nathanson1974}
M. B. Nathanson,
\emph{Minimal bases and maximal nonbases in additive number theory},
J. Number Theory \textbf{6} (1974), 324--333.

\bibitem{Nathanson1988}
M. B. Nathanson,
\emph{Minimal bases and powers of $2$},
Acta Arith. \textbf{49} (1988), 525--532.

\bibitem{Pilatte2024}
C. Pilatte,
\emph{A solution to the Erd\H{o}s--S\'ark\H{o}zy--S\'os problem on asymptotic Sidon bases of order $3$},
Compos. Math. \textbf{160} (2024), 1418--1432.

\bibitem{Sun2021}
C.-F. Sun,
\emph{On a problem of Nathanson on minimal asymptotic bases},
J. Number Theory \textbf{218} (2021), 152--160.

\bibitem{Vu2000}
V. H. Vu,
\emph{On a refinement of Waring's problem},
Duke Math. J. \textbf{105} (2000), 107--134.

\bibitem{Wooley2003}
T. D. Wooley,
\emph{On Vu's thin basis theorem in Waring's problem},
Duke Math. J. \textbf{120} (2003), 1--34.

\end{thebibliography}
\end{document}